\documentclass[12pt]{amsart}
\usepackage{amssymb}
\usepackage{graphicx}
\usepackage[T1]{fontenc}

\usepackage{amssymb}
\usepackage{graphicx}

\makeatletter
\def\imod#1{\allowbreak\mkern5mu({\operator@font mod}\,\,#1)}
\makeatother

\newcommand{\nin}{\notin}

\newcommand{\cross}{\times}

\newcommand{\ceil}[1]{\( \lceil #1 \) \rceil}

\renewcommand{\(}{\left}
\renewcommand{\)}{\right}

\renewcommand{\H}{\mathcal{H}}

\newcommand{\R}{\mathbb{R}}
\newcommand{\Z}{\mathbb{Z}}

\newcommand{\e}{\varepsilon}
\renewcommand{\t}{\theta}

\newcommand{\G}{{\ensuremath \mathcal{G}}}
\newcommand{\0}{{\sf 0}}
\newcommand{\1}{{\sf 1}}
\newcommand{\s}{{\ensuremath \sigma}}

\newcommand{\tbin}[2]{{\ensuremath\textstyle \binom{#1}{#2}}}
\newcommand{\dbin}[2]{{\ensuremath\displaystyle \binom{#1}{#2}}}
\newcommand{\cc}{{\ensuremath \mathcal{C}}}
\newcommand{\prim}[2]{\text{prim}_{#1}(#2)}
\newcommand{\pattern}[1]{[#1]}
\newcommand{\V}[1]{\mbox{\boldmath{$#1$}}}
\newcommand{\integral}[4]{\int_{#1}^{#2} #3\; {\rm d} #4}

\renewcommand{\L}{L^2[0,1]}
\newcommand{\LR}{L_{\R}^2[0,1]}

\theoremstyle{plain}
\newtheorem{thm}{Theorem}
\newtheorem{corol}[thm]{Corollary}
\newtheorem{lemma2}[thm]{Lemma}
\newtheorem{prop}[thm]{Proposition}

\theoremstyle{definition}
\newtheorem*{defn}{Definition}
\newtheorem*{eg}{Example}
\newtheorem*{note}{Note}
\newtheorem*{obs}{Observation}

\begin{document}

\title{Control of cancellations that restrain the growth of a binomial
       recursion}
\author{Magnus Aspenberg}
\author{Rodrigo P\'erez}
\address{LD-224R IUPUI, 402 N. Blackford St., Indianapolis, IN 46202, USA}
\email{rperez@math.iupui.edu}
\thanks{The second author was supported by NSF grant DMS-0701557}

\begin{abstract}
  We study a recursion that generates real sequences depending on a parameter
  $x$. Given a negative $x$ the growth of the sequence is very difficult to
  estimate due to canceling terms. We reduce the study of the recursion to a
  problem about a family of integral operators, and prove that for every
  parameter value except $-1$, the growth of the sequence is factorial. In the
  combinatorial part of the proof we show that when $x=-1$ the resulting
  recurrence yields the sequence of alternating Catalan numbers, and thus has
  exponential growth. We expect our methods to be useful in a variety of
  similar situations.
\end{abstract}

\maketitle


{\renewcommand{\l}{{\ensuremath \lambda}}

\section{Introduction}
Fix an arbitrary real number $x \neq 0$, and consider the sequence defined by
the recursive expression
\begin{equation}
 \label{eqn:Formula}
  a_1 = x \quad , \qquad
  a_n = x \sum_{r=\ceil{\frac{n}{2}}}^{n-1} \dbin{r}{n-r} a_r.
\end{equation}

For $x = 1$, formula~\eqref{eqn:Formula}
produces the sequence $1, 2, 7, 34, 214, 1652, \ldots$ Note that the last
summand $\tbin{n-1}{1} a_{n-1}$ guarantees that $a_n > (n-1)!$ This means that
$\{ a_n \}$ grows very fast since $\displaystyle n!  > (n/e)^n$. We prove

\begin{thm}
 \label{thm:Superexponential}
  For any real $x \neq -1, 0$ the sequence $\{ a_n \}$ defined by
  \eqref{eqn:Formula} grows super-exponentially.
\end{thm}

This is an interesting behavior, and not altogether obvious because when $x <
0$, there are a lot of cancellations. In fact, when $x = -1$, the positive and
negative terms exactly balance out to yield a surprising contrast:

\begin{thm}
 \label{thm:Catalan}
  When $x = -1$, formula~\eqref{eqn:Formula} produces the sequence $(-1)^n
  C_n$ of Catalan numbers with alternating signs, and therefore grows
  exponentially.
\end{thm}

The phenomenon at play is very interesting. In the first part of the paper
some combinatorial constructions will allow us to prove the results when $x >
0$ and when $x \leq -1$. However, when $x \in (-1,0)$, the cancellations and
the small size of $x$ conspire to render elementary arguments ineffective. In
the second part we introduce functional analytic methods to control the effect
of the cancellations in that case. We expect these ideas to be useful in a
variety of similar situations.

\subsection{Structure}
The paper has two parts. In the first (sections \ref{sect:Combinatorics} to
\ref{sect:FirstProofs}) we prove Theorem~\ref{thm:Catalan} and the case $x
\nin [-1,0]$ of
Theorem~\ref{thm:Superexponential}. Section~\ref{sect:Combinatorics} presents
some basic facts about hypercube graphs and the Catalan numbers;
Section~\ref{sect:Structure-a_n} defines the combinatorial structure we use;
and Section~\ref{sect:FirstProofs} contains the proofs.

The second part (sections \ref{sect:xIn(-1,0)} to
\ref{sect:Proof-Norms1AndInfty}) tackles the case $x \in (-1,0)$ of
Theorem~\ref{thm:Superexponential}. Section~\ref{sect:xIn(-1,0)} uses the
combinatorial knowledge gained in the first part to derive an alternative
expression for $a_n$ as the sum of a sequence of numbers $S_n(1), \ldots,
S_n(n-1)$ constructed recursively. This sequence is translated into a function
$s_n \in \L$, and the recursion is interpreted as an integral
operator. Section~\ref{sect:FunctionalApproach} contains the proof of the
theorem assuming the statement of Lemma~\ref{lemma:Norms1AndInfty}, and
Section~\ref{sect:Proof-Norms1AndInfty} is devoted to the proof of
Lemma~\ref{lemma:Norms1AndInfty}.

\section{Basic Combinatorial Facts}
\label{sect:Combinatorics}
\subsection{Hypercubes}
\label{sect:Hypercubes}
The {\em hypercube graph} $\H_n$ is the graph whose set of vertices $V_n$
consists of all $n$-vectors with coordinates $\0$ or $\1$. Two
vertices are adjacent whenever they differ in one coordinate. There is a
natural stratification of $V_n$ by the number of coordinates of each
value in a vertex; accordingly, let $V_{n,r} \subset V_n$ denote the vertices
with $r$ coordinates equal to $\1$. There are other equivalent definitions of
hypercube graphs. The advantage of the definition in terms of
binary coordinates is that the following facts become obvious; compare
Figure~\ref{fig:Hypercubes}.

\renewcommand{\theenumi}{H\arabic{enumi}}
\begin{enumerate}
  \item \label{item:HC-Vertices}
        $|V_n| = 2^n$.
  \item \label{item:HC-Stratification}
         $\displaystyle |V_{n,r}| = \tbin{n}{r}$.
  \item \label{item:HC-CrossProduct}
        If $n = m_1 + m_2$, then $\H_n = \H_{m_1} \cross \H_{m_2}$.
\end{enumerate}
Incidentally, items \eqref{item:HC-Vertices} and
\eqref{item:HC-Stratification} give a succinct proof of the binomial
identity $\sum_{r=0}^n \tbin{n}{r} = 2^n$. If instead of just counting
vertices, they are assigned weight $x^j$, item~\eqref{item:HC-CrossProduct}
furnishes a recursive proof of Newton's binomial formula
\begin{enumerate}
  \setcounter{enumi}{3}
  \item \label{item:HC-Binomial}
        $\displaystyle \sum_{j=0}^n \tbin{n}{r} x^j = (1+x)^n$.
\end{enumerate}

\begin{figure}[h]\refstepcounter{figure}\addtocounter{figure}{-1}
 \label{fig:Hypercubes}
  \begin{center}
    \includegraphics{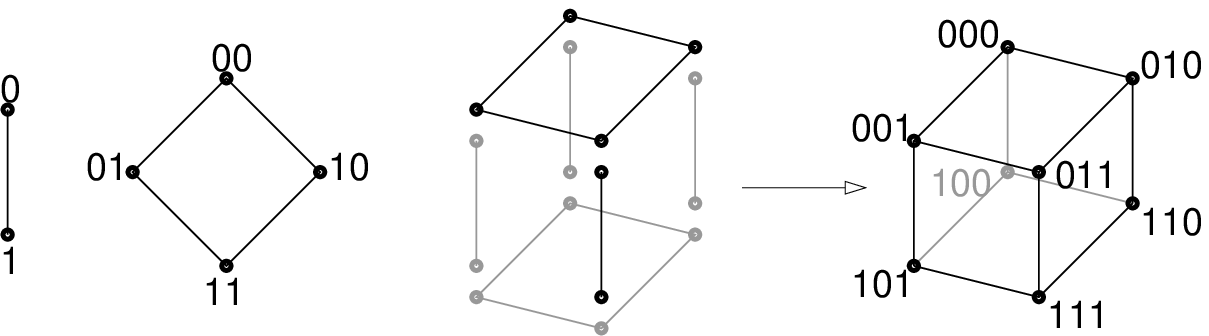}
    \caption{The hypercube graphs $\H_1$, $\H_2$, $\H_3$, and a decomposition
             of the latter as $\H_1 \cross \H_2$.}
  \end{center}
\end{figure}

\subsection{Catalan Numbers and Lattice Paths}\label{sect:Catalan}
The Catalan numbers $1, 1, 2, 5, 14, 42, 132, \ldots$ \cite[{\sf
A000108}]{OEIS} are defined by the formula
\[C_{n+1} = \frac{\dbin{2n}{n}}{(n+1)}.\]


The exponential rate of growth of the sequence $\{ C_n \}$ follows easily from
Stirling's formula:
\[
  C_{n+1} =
  \frac{(2n)!}{(n!)^2(n+1)} \sim
  \frac{\sqrt{2 \pi (2n)} \( ( \frac{2n}{e} \) )^{2n}}
       {\( ( \sqrt{2 \pi n} \( ( \frac{n}{e} \) )^n \) )^2 (n+1)} =
  \frac{2^{2n}}{\sqrt{\pi n} (n+1)} \sim
  \frac{4^n}{\sqrt{\pi}\, n^{3/2}}.
\]

\begin{defn}
  A {\em lattice path} is a path in the lattice $\Z \cross \Z$ that moves one
  horizontal or vertical unit at every step without self-intersections. We
  consider {\em monotone paths}, which never move left nor down. Note that a
  monotone path from $(0,0)$ to $(m,n)$ requires $m+n$ steps. Choosing one
  such path is tantamount to deciding which of these steps will be the $m$
  horizontal steps, so the number of monotone lattice paths from $(0,0)$ to
  $(m,n)$ is $\tbinom{m+n}{m}$.
\end{defn}



\begin{lemma2}
 \label{lemma:CountMonotonePaths}
  The number of monotone paths from $(0,0)$ to $(n,n)$ that do not cross over
  the diagonal $\{ y = x \}$ is equal to $C_n$.
\end{lemma2}

\begin{figure}[h]\refstepcounter{figure}\addtocounter{figure}{-1}
 \label{fig:LatticePaths}
  \begin{center}
    \includegraphics{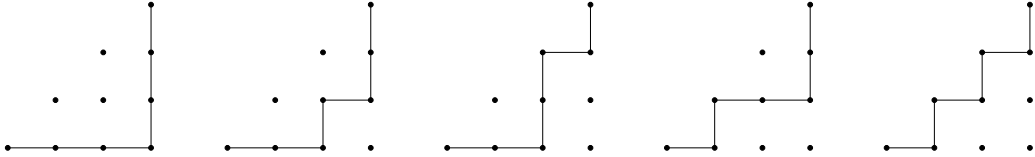}
    \caption{The $C_3 = 5$ monotone paths from $(0,0)$ to $(3,3)$.}
  \end{center}
\end{figure}


\begin{proof}[Proof of Lemma~\ref{lemma:CountMonotonePaths}]
  A monotone path $\gamma$ from $(0,0)$ to $(n,n)$ that crosses over the
  diagonal will pass through a point $(j,j+1)$. Let $P$ be the first such
  point, and $\gamma'$ the portion of $\gamma$ going from $P$ to
  $(n,n)$. Reflecting $\gamma'$ on the diagonal $\{ y = x+1 \}$ transforms
  $\gamma$ into a monotone path from $(0,0)$ to $(n-1,n+1)$. This operation is
  bijective because such paths {\em must} cross over the diagonal. Therefore
  the number of monotone paths from $(0,0)$ to $(n,n)$ that do not cross over
  the diagonal equals the number of all monotone paths from $(0,0)$ to
  $(n,n)$, minus the number of monotone paths from $(0,0)$ to $(n-1,n+1)$;
  i.e.,
  \[
    \dbin{2n}{n} - \dbin{2n}{n-1} =
    \dbin{2n}{n} - \frac{n}{n+1} \dbin{2n}{n} =
    \frac{\dbin{2n}{n}}{n+1}. \qedhere
  \]
\end{proof}

\section{Structure of $a_n$}\label{sect:Structure-a_n}
\subsection{Signatures and Binomial Products}\label{sect:SignaturesAndBroducts}
The internal structure of the expression $a_n$ is better understood by
separating the different contributions of weight $x^r$. After expanding the
recursive expressions in \eqref{eqn:Formula}, the first few terms are
\begin{align}
 \label{eqn:ExpandedFormula}
  a_1 &=&=& x, \\
  \notag
  a_2 &= x \( [ \tbin{1}{1} a_1 \) ] &=& \tbin{1}{1} x^2, \\
  \notag
  a_3 &= x \( [ \tbin{2}{1} a_2 \) ] &=& \tbin{2}{1} \tbin{1}{1} x^3, \\
  \notag
  a_4 &= x \( [ \tbin{2}{2} a_2 + \tbin{3}{1} a_3 \) ] &=&
           \tbin{2}{2} \tbin{1}{1} x^3 +
           \tbin{3}{1} \tbin{2}{1} \tbin{1}{1} x^4, \\
  \notag
  a_5 &= x \( [ \tbin{3}{2} a_3 + \tbin{4}{1} a_4 \) ] &=&
           \tbin{3}{2} \tbin{2}{1} \tbin{1}{1} x^4 +
           \tbin{4}{1} \tbin{2}{2} \tbin{1}{1} x^4 + \\
  \notag
      &&&  \qquad \tbin{4}{1} \tbin{3}{1} \tbin{2}{1} \tbin{1}{1} x^5.
\end{align}

This symbolic manipulation makes it clear that $a_n$ is the sum of all
products of the form
\begin{equation}
 \label{eqn:ContributionBySignature}
  \dbin{b_s}{n-b_s} \dbin{b_{s-1}}{b_{s-1}-b_{s-2}} \ldots
  \dbin{b_2}{b_2-b_1} \dbin{b_1}{1} \cdot x^{s+1},
\end{equation}
such that
\begin{equation}
 \label{eqn:SignatureConditions}
  n =: b_{s+1} > b_s > b_{s-1} > \ldots > b_1 = 1\,, \quad
  \text{and}\quad
  b_{j+1} \leq 2b_j {\ } (j=1, \ldots, s).
\end{equation}
The last condition is a consequence of the fact that the sum in
\eqref{eqn:Formula} starts at $r = \ceil{\frac{n}{2}}$. Note that this
condition forces $b_2 = 2$.

\begin{defn}
  A tuple $\s = (n, b_s, \ldots, b_1)$ satisfying
  \eqref{eqn:SignatureConditions} is called an {\em $n$-signature}. The
  $n$-signature that contains all the numbers from
  1 to $n$ is called {\em canonical}.
\end{defn}

\begin{note}
  Compare the recursion \eqref{eqn:Formula} with the similar looking $\alpha_1
  = x$, $\alpha_n = x \sum_{r=\ceil{n/2}}^{n-1} \alpha_r$ in which the
  binomial coefficients have been removed. From the above discussion we see
  that $\alpha_n$ is the sum of weights $x^s$, taken over all signatures $(n,
  b_s, \ldots, b_1)$.  Replacing $x$ with 1 shows that the number of distinct
  $n$-signatures is given by the recursion
  \[
    N_1 = 1 \quad , \qquad
    N_n = \sum_{r=\ceil{\frac{n}{2}}}^{n-1} N_r.
  \]
  The numbers $\{ N_n \} = \{ 1, 1, 1, 2, 3, 6, 11, 22, 42, \ldots \}$ form
  the Narayana-Zidek-Capell sequence \cite[{\sf A002083}]{OEIS}.
\end{note}

\subsection{Arrays and blocks}\label{sect:ArraysAndBlocks}
When faced with an expression made of binomial coefficients, the natural thing
to ask is what kind of combinatorial object is being counted. To a given
signature $\s = (n, b_s, \ldots, b_1)$ we will assign a {\em tower} that can
be filled with an {\em array} of numbers in exactly $\tbin{b_s}{n-b_s} \ldots
\tbin{b_1}{1}$ ways.

\begin{defn}
  Given $\s = (n, b_s, \ldots, b_1)$, consider a {\em tower} of $n-1$ square
  cells split into {\em blocks} of lengths $(n-b_s), (b_s-b_{s-1}), \ldots,
  (b_3-b_2), (b_2-b_1)$ from top to bottom as in Figure~\ref{fig:Towers}. The
  {\em position} of a block is the height $b_j$ of its lowest cell, so the
  signature condition $b_{j+1} \leq 2b_j$ implies that a block is never taller
  than its position. An {\em array} associated to $\s$ is an assignment of
  numbers to every cell in the tower of $\s$ such that the numbers in the
  $j^{\rm th}$ block (at position $b_j$) are chosen from the set $\{ 1, 2,
  \ldots, b_j \}$ and appear in descending order. An array associated to the
  canonical signature is also called {\em canonical}.
\end{defn}

\begin{figure}[h]\refstepcounter{figure}\addtocounter{figure}{-1}
 \label{fig:Towers}
  \begin{center}
    \includegraphics{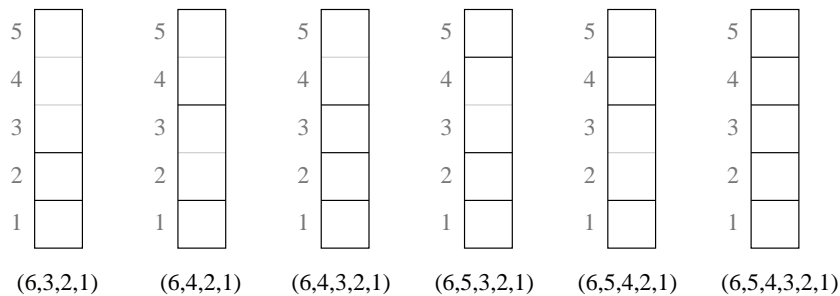}
    \caption{The towers associated to all 6-signatures. The second tower for
             instance, has blocks at positions 1, 2, and 4. The rightmost
             signature is the canonical one.}
  \end{center}
\end{figure}

\begin{figure}[h]\refstepcounter{figure}\addtocounter{figure}{-1}
 \label{fig:Arrays}
  \begin{center}
    \includegraphics{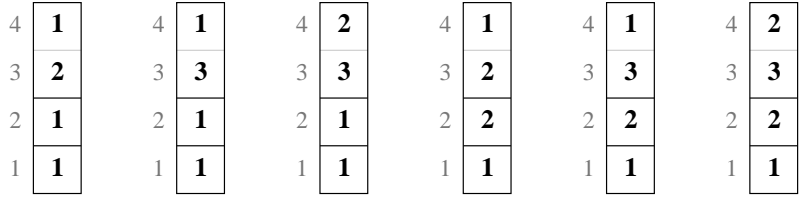}
    \caption{The 5-signature $(5,3,2,1)$ has six associated arrays: The block
    at position 2 can hold either a 1 or a 2, while the block at position 3
    can hold any descending combination of the numbers $1,2,3$.}
  \end{center}
\end{figure}

\begin{lemma2}
 \label{lemma:SignatureNumbers}
  Let $\s = (n,b_s, \ldots, b_1)$ be an $n$-signature. Then
  \renewcommand{\theenumi}{\alph{enumi}}
  \begin{enumerate}
    \item \label{item:ArrayCount}
          The number of $n$-arrays associated to $\s$ is
          $\tbin{b_s}{n-b_s} \ldots \tbin{b_1}{1}$.
    \item \label{item:CanonicalArrayCount}
          The number of canonical $n$-arrays is $(n-1)!$
    \item \label{item:TotalArrayCount}
          The total number of $n$-arrays is given by
          formula~\eqref{eqn:Formula} when $x = 1$.
  \end{enumerate}
\end{lemma2}

\begin{proof}
  The tower associated to $\s$ has $s$ blocks. The $j^{\rm th}$ block is based
  at position $b_j$ and its length is $b_{j+1} - b_j$ (for the topmost block
  the length is $n - b_s$). Therefore the $j^{\rm th}$ block can be filled
  with an arbitrary choice of $b_{j+1} - b_j$ numbers between 1 and $b_j$;
  i.e., $\tbin{b_j}{b_{j+1} - b_j}$ possibilities. This proves
  \eqref{item:ArrayCount}, from which item \eqref{item:CanonicalArrayCount}
  follows immediately. To prove \eqref{item:TotalArrayCount}, note from
  \eqref{eqn:ContributionBySignature} that $a_n$ counts $n$-arrays with weight
  $x^s$. Thus, when $x=1$, $a_n$ simply counts the number of $n$-arrays as
  claimed.
\end{proof}

\begin{defn}
  The sequence of numbers that specifies an array is called a {\em
  pattern}. We convene to read patterns from the bottom up; thus, for
  instance, the rightmost array in Figure~\ref{fig:Arrays} has pattern
  $\pattern{1232}$.
\end{defn}

\begin{lemma2}
  A tuple $\pattern{t_1, \ldots, t_{n-1}}$ is a valid pattern if and only if
  \[
    t_j \leq j \text{ for all } 1 \leq j \leq n-1.
  \]
\end{lemma2}

\begin{proof}
  In a canonical array every block has length one. This means that the
  position of the $j^{\rm th}$ block is $b_j = j$, and the number in this
  block is $t_j$. Thus, in this case, the pattern condition is equivalent to
  $t_j \leq b_j = j$, proving the result for canonical arrays.

  In a non-canonical array, the cell at position $j$ belongs to a block at
  position $i \leq j$. The pattern condition states that the number $t_j$ in
  that cell must be at most $i$, and the result follows.
\end{proof}

\subsection{Array Hypercubes}\label{sect:ArrayHypercubes}
\begin{defn}
  If an array has a block at position $p$ with more than one cell, the block
  can be split into two shorter blocks. The result is a valid array since
  the blocks have positions $p$ and $p+\eta > p$ ($\eta$ is the location of
  the split within the original block), and the numbers in both blocks are all
  at most $p$. We call this operation on arrays a {\em split}. Note that an
  array can usually be split in several ways, all of which commute. Moreover,
  repeated splitting eventually results in a canonical array.

  The reverse operation is also well defined. If an array has two consecutive
  blocks at positions $p$ and $p+\eta$, and the numbers contained in both
  blocks run together in descending order, the two blocks can be combined into
  a single one. This is because the new block is at position $p$ and contains
  numbers in descending order, which means that the length of the new block
  cannot exceed its position. In other words,
  condition~\eqref{eqn:SignatureConditions} is satisfied. This operation on
  arrays is called a {\em merge}. As with splitting, merge operations are
  commutative, and repeated merging must terminate. An array where no pair of
  blocks can be merged is called {\em primitive}.
\end{defn}

\begin{defn}
  The graph $\G_n$ on the set of $n$-arrays is defined by joining any two
  arrays related by a single split/merge operation. Arrays belong to the same
  connected component of $\G_n$ when they have the same pattern of numbers
  (disregarding block divisions). Note that a split/merge is possible at a
  given position if and only if the numbers at that position are in descending
  order. In particular, splitting/merging does not depend on the structure of
  blocks in an array, but only on the pattern of numbers. This yields the
  following lemma.
\end{defn}

\begin{lemma2}
 \label{lemma:ArrayHypercubes}
  Every connected component of $G_n$ is homeomorphic to a hypercube graph.
\end{lemma2}

\begin{proof}
  Consider an array $A \in \G_n$. The connected component $\cc$ of $A$
  consists of all arrays with the same pattern of numbers as $A$. This pattern
  has $\ell$ {\em descents} (locations where the numbers are in descending
  order). Now view such locations as placeholders for a symbol $\1$ or $\0$
  depending on whether two blocks of $A$ meet at that location or not. This
  puts the arrays of $\cc$ in correspondence with vertices of the hypercube
  graph $\H_{\ell}$; see Figure~\ref{fig:ArrayHypercube}. Since a split/merge
  depends only on the pattern of numbers, all edges of $\H_{\ell}$ are
  included and $\cc$ is homeomorphic to $\H_{\ell}$.
\end{proof}

\begin{obs}
  Every hypercube $\cc \subset \G_n$ has a unique primitive array and a unique
  canonical array. In particular, the numbers of hypercubes in $\G_n$ and of
  primitive $n$-arrays are both equal to $(n-1)!$ Also, if $\ell$ is as in the
  proof above, the primitive array has $s = n - \ell - 1$ blocks (because the
  canonical array has $n-1$), so $\cc$ is homeomorphic to $\H_{n-s-1}$.
\end{obs}

\begin{figure}[h]\refstepcounter{figure}\addtocounter{figure}{-1}
 \label{fig:ArrayHypercube}
  \begin{center}
    \includegraphics{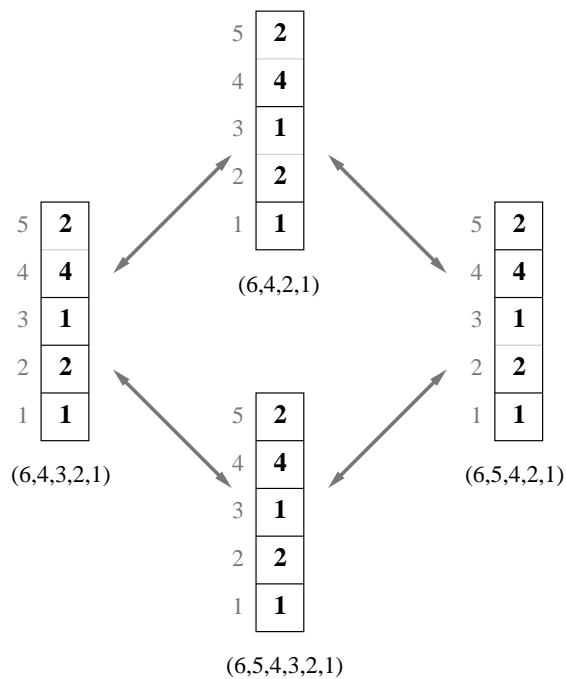}
    \caption{The pattern $\pattern{12142}$ has 2 descents, and thus determines
             four \mbox{6-arrays} connected by split/merge operations into a
             square $\H_2$.}
  \end{center}
\end{figure}

\section{The First Proofs}\label{sect:FirstProofs}
Each $a_n$ is a polynomial in $x$. When written with the monomials $\xi_r x^r$
in ascending order by degree, we say $a_n$ is in {\em basic format}. Since
every $n$-array with $r-1$ blocks contributes $x^r$ to the total $a_n$, the
coefficient $\xi_r$ counts the number of such arrays.

\begin{eg}
  The first few $a_n$ in basic format are (compare
  \eqref{eqn:ExpandedFormula}):
  \begin{align}
   \label{eqn:BasicFormats}
    \notag a_1 &= x, \\
    \notag a_2 &= x^2, \\
    \notag a_3 &= 2 x^3, \\
    \notag a_4 &= x^3 + 6 x^4,  \\
    \notag a_5 &= 10 x^4 + 24 x^5, \\
           a_6 &= 8 x^4 + 86 x^5 + 120 x^6.
  \end{align}

  When $n = 6$ for instance, we know that the towers with 4 blocks correspond
  to the three signatures $(6,4,3,2,1)$, $(6,5,3,2,1)$, and $(6,5,4,2,1)$ (see
  Figure~\ref{fig:Towers}). These have
  $\tbin{4}{2}\tbin{3}{1}\tbin{2}{1}\tbin{1}{1} = 36$,
  $\tbin{5}{1}\tbin{3}{2}\tbin{2}{1}\tbin{1}{1} = 30$, and
  $\tbin{5}{1}\tbin{4}{1}\tbin{2}{2}\tbin{1}{1} = 20$ associated arrays
  respectively, and we see that each of these 86 arrays contributes $x^5$ to
  the value of $a_6$.
\end{eg}

Recall that the array graph $\G_n$ consists of isolated hypercube components.
We can also break $a_n$ down as a sum of contributions by hypercubes. Let $A$
be a primitive $n$-array with $s = r-1$ blocks. We know that $A$ contributes
$x^r$ to $a_n$. Since the connected component $\cc$ of $\G$ containing $A$ is
homeomorphic to $\H_{n-r}$, items \eqref{item:HC-Stratification} and
\eqref{item:HC-Binomial} in Section~\ref{sect:Hypercubes} give
\begin{equation}
 \label{eqn:HypercubeContribution}
  \sum_{a \in \cc} x^{(\#\text{blocks of }a)+1}
    \stackrel{\eqref{item:HC-Stratification}}{=}
  \sum_{j=0}^{n-r} \dbin{n-r}{j} x^{r+j}
    \stackrel{\eqref{item:HC-Binomial}}{=}
  x^r(1+x)^{n-r}.
\end{equation}

Let $\prim{n}{r}$ be the number of primitive arrays with $r-1$ blocks; this is
also the number of components of $\G_n$ homeomorphic to
$\H_{n-r}$. Equation~\eqref{eqn:HypercubeContribution} shows that the total
contribution to $a_n$ of all arrays in all such hypercubes is $\prim{n}{r}
\cdot x^r(1+x)^{n-r}$, so

\begin{equation}
 \label{eqn:BinomialDecomposition}
  a_n =
  \sum_{r=\ceil{\log_2 n}}^n \prim{n}{r} \cdot x^r(1+x)^{n-r}
\end{equation}
(Condition~\eqref{eqn:SignatureConditions} implies that the least possible
number of blocks is $\ceil{\log_2 n}$).

When $a_n$ is written in that form, we can read that each of the $\prim{n}{r}$
components of $\G_n$ with dimension $n - r$ contributes $x^r(1+x)^{n-r}$ to
the total sum $a_n$. We will say that $a_n$ is in {\em binomial format}.

\begin{eg}
  Let us compute the binomial format of $a_6$. From \eqref{eqn:BasicFormats}
  we see that there are 8 arrays with 3 blocks. All of these must be
  primitive, so $\prim{6}{4} = 8$. These arrays (together with those obtained
  by splitting) determine 8 copies of $\H_2$ in $\G_6$. Gathering the
  monomials of all arrays in these hypercubes gives
  \begin{multline*}
    a_6 = 8 x^4 + 86 x^5 +120 x^6 = \\
    (8 x^4 + 16 x^5 +8 x^6) + 70 x^5 +112 x^6 = \\
    8 x^4(1+x)^2 + 70 x^5 + 112 x^6.
  \end{multline*}
  The remaining 70 arrays with 4 blocks must be primitive since they do not
  belong to an $\H_2$; i.e., $\prim{6}{5} =70$. These arrays (together with
  those obtained by splitting) determine 70 copies of $\H_1$ in $\G_6$. Thus,
  \[
    8 x^4(1+x)^2 + 70 x^5 + 112 x^6 =
    8 x^4(1+x)^2 + 70 x^5(1+x)^1 + 42 x^6.
  \]
  The first few $a_n$ in binomial format are:
  \begin{align*}
    \notag a_1 &= x^1 y^0, \\
    \notag a_2 &= x^2 y^0, \\
    \notag a_3 &= 2 x^3 y^0, \\
    \notag a_4 &= x^3 y^1 + 5 x^4 y^0, \\
    \notag a_5 &= 10 x^4 y^1 + 14 x^5 y^0, \\
           a_6 &= 8 x^4 y^2 + 70 x^5 y^1 + 42 x^6 y^0,
  \end{align*}
  where $y$ stand for $1+x$. This convetion makes for cleaner looking
  expressions, and {\em will be consistently used in the rest of the paper}.
\end{eg}

Now that the combinatorial structure is in place, the proofs of
Theorem~\ref{thm:Catalan} and the case $x \nin [-1,0]$ of
Theorem~\ref{thm:Superexponential} are straightforward.

\begin{proof}[Proof of Theorem \ref{thm:Catalan}]
  Since $x = -1$, the only non-zero term in \eqref{eqn:BinomialDecomposition}
  occurs when $r = n$. In other words,
  \[a_n = (-1)^n \prim{n}{n}.\]
  Now, $\prim{n}{n}$ is the number of $n$-arrays with $n-1$ blocks; that is,
  arrays that are both primitive and canonical. By definition, these are
  arrays with non-decreasing patterns of $n-1$ numbers. Any such pattern can
  be associated to a monotone lattice path from $(0,0)$ to $(n-1, n-1)$ that
  does not cross over the diagonal. Simply substract 1 from each entry in the
  pattern and interpret the results as heights of the horizontal steps of a
  monotone path (compare Figure~\ref{fig:LatticePaths}). This procedure is
  bijective, so by Lemma~\ref{lemma:CountMonotonePaths}, $\prim{n}{n} = C_n$.
\end{proof}

\begin{proof}
 [Proof of Theorem~\ref{thm:Superexponential} when {$x \nin [-1,0]$}]

  Depending on the sign of $x$, one of the two formats for $a_n$ displays no
  cancellations.

  \smallskip\noindent\underline{$x > 0$}: All monomials $\xi_r x^r$ in the
  basic format of $a_n$ are positive because the coefficient $\xi_r$ counts
  arrays with $r-1$ blocks. In particular, $a_n$ is larger than the highest
  order monomial. The coefficient of this monomial is the number of arrays
  with the most blocks; i.e., canonical arrays. By
  Lemma~\ref{lemma:SignatureNumbers},
  \[a_n > (n-1)! \cdot x^n.\]

  \smallskip\noindent\underline{$x < -1$}: Since $y<0$, all terms $\prim{n}{r}
  \cdot x^r y^{n-r}$ in the binomial format of $a_n$ have the same sign, and
  do not cancel each other. Also, $|x| > |y|$, so
  \[
    |a_n| =
    \sum_r \prim{n}{r} \cdot |x^r y^{n-r}| >
    \sum_r \prim{n}{r} \cdot |y^n| =
    (n-1)! \cdot |y^n|,
  \]
  where the last equality follows from the observation after
  Lemma~\ref{lemma:ArrayHypercubes}. \\

  In both cases, $|a_n|$ is larger than $(n-1)! \cdot w^n$ for some positive
  $w$, and the result holds.
\end{proof}

\section{The case $x \in (-1,0)$}\label{sect:xIn(-1,0)}
The situation when $x \in (-1,0)$ is more delicate because the terms in both
the basic and binomial formats of $a_n$ have alternating signs. Our strategy
in this second part involves a different representation
(\eqref{eqn:Recover-a_n} and \eqref{eqn:PatternRecursion}) of $a_n$. We will
interpret the sequences $\{ S_n(1), \ldots S_n(n-1) \}$ as functions $s_n \in
\L$ and the recursion \eqref{eqn:PatternRecursion} as a sequence of integral
operators $A_n:s_n \mapsto s_{n+1}$. We will deduce some facts about the shape
of the graph of $s_n$, and about the limit operator $T = \lim A_n$. Then we
will use this information to show that the largest eigenvalue $\l$ of $T$
bounds from below the exponential rate of decay of $a_n/(n-2)!$ \\

\subsection{A new recursion}
So far we have established that $a_n$ is the sum of contributions of the form
$x^r$ over a large set of arrays (for each array, $r-1$ is the number of
blocks). We grouped arrays with the same number pattern into hypercube graphs,
and showed that $a_n$ is the sum of contributions $x^{n-\ell} (1+x)^\ell =
x^{n-\ell} y^\ell$ over hypercubes $\cc$, where $\ell$ is the dimension of
each $\cc$. This dimension is the number of descents in the associated
pattern, so we can abandon arrays and express $a_n$ directly as a sum of
contributions over patterns:
\[
  a_n =
  \sum_{\ell=0}^{n - \ceil{\log_2 n} - 1}
  \sum_{\substack{\text{patterns with} \\ \text{$\ell$ descents}}}
    x^{n-\ell} y^\ell.
\]
This allows us to sort the contributions to $a_n$ made by individual
patterns. To this end, consider an $n$-pattern $\pi = \pattern{t_1, \ldots,
t_{n-1}}$. If the truncated pattern $\pattern{t_1, \ldots, t_{n-2}}$
contributes $x^a y^b$ to $a_{n-1}$, then $\pi$ contributes $x^a y^{b+1}$ or
$x^{a+1} y^b$ to $a_n$ depending on whether $t_{n-1} < t_{n-2}$ or not (i.e.,
on whether $\pi$ has one extra descent or not at the last position). This
motivates the following definition.

\begin{defn}
  For $n \geq 2$ let $S_n(r)$ denote the sum of contributions of all patterns
  $\pattern{t_1, \ldots, t_{n-1}}$ such that $t_{n-1}$ equals $r$ (thus, $r$
  can take values in $\{ 1, \ldots, n-1 \}$). In particular, $S_2(1) = x^2$,
  and
  \begin{equation}
   \label{eqn:Recover-a_n}
    a_n = \sum_{j=1}^{n-1} S_n(j).
  \end{equation}
\end{defn}
By the previous argument, $S_{n+1}(r)$ can be computed from the contributions
of $n$-patterns:
\begin{equation}
 \label{eqn:PatternRecursion}
  S_{n+1}(r) =
  x \cdot \sum_{j=1}^r S_n(j) +
  y \cdot \sum_{j=r+1}^{n-1} S_n(j).
\end{equation}
(when $r = n-1$ or $n$, there is no descent in the last position, so
\eqref{eqn:PatternRecursion} should be interpreted to mean $S_{n+1}(n-1)
= S_{n+1}(n) = x \cdot \sum_{j=1}^{n-1} S_n(j)$).

\begin{figure}[h]\refstepcounter{figure}\addtocounter{figure}{-1}
 \label{fig:SequenceSample}
  \begin{center}
    \includegraphics[width=4in]{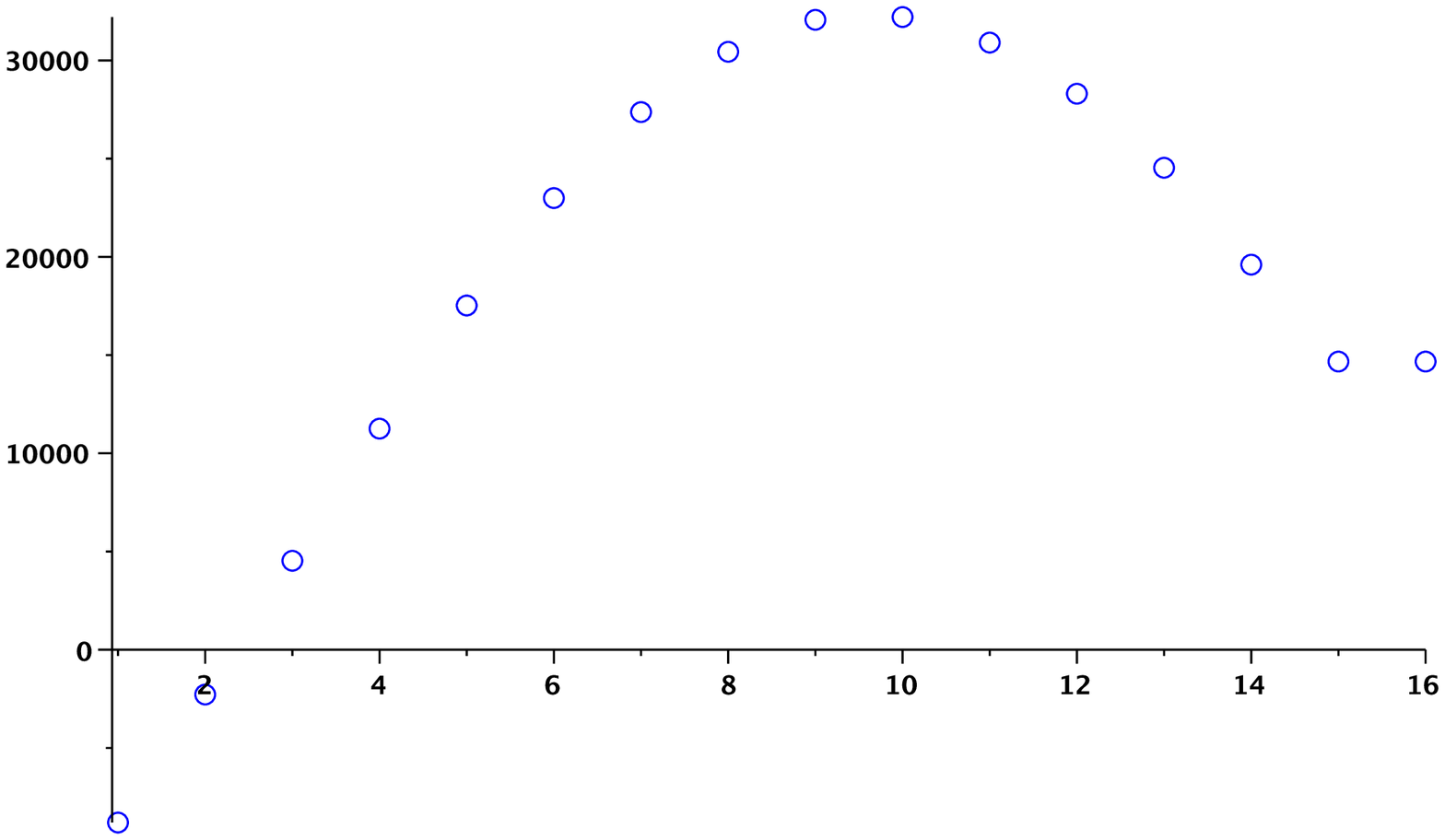}
    \caption{}
  \end{center}
\end{figure}

Figure~\ref{fig:SequenceSample} shows a plot of the values $\{ S_{16}(1),
\ldots, S_{16}(15) \}$ when $x = -1/2$. It is not coincidental that the graph
looks sinusoidal.

\subsection{Sinusoidal shape of \boldmath{$S_n$}}
Note that
\begin{equation}
 \label{eqn:BasicSequenceRelation}
  S_{n+1}(r) = S_{n+1}(r-1) - S_n(r).
\end{equation}
From this relation we can derive a more convenient method of computing the
sequence $S_n$:
\begin{itemize}
  \item $S_2(1) = x^2$; \text{ and for } $n \geq 3$,
  \item $S_n(0) = y \sum_{j=1}^{n-2} S_{n-1}(j)$,
  \item $S_n(r) = S_n(r-1) - S_{n-1}(r),\ (1 \leq r \leq n-2)$,
  \item $S_n(n-1) = S_n(n-2)$.
\end{itemize}

\begin{obs}
  $S_n(0)$ is not part of the original sequence, but we will find it easier to
  study the properties of $S_n$ by including this auxiliary term in the
  discussion. For instance, note that $S_n(0) = ya_{n-1}$ and $S_n(n-1) = x
  a_{n-1}$. Recall that $y = 1+x$ and $x \in (-1,0)$. {\em It is vital to the
  coming arguments that these two values have opposite signs}.
\end{obs}

\begin{defn}
  The sequence $S_n$ has
  \begin{itemize}
    \item a {\em sign change} if $n \geq 4$ and there are $a,b$ ($1 \leq a < b
          \leq n-2$) such that
          \[\begin{array}{rcccl}
            S_n(a) &<& 0 = S_n(a+1) = \ldots = S_n(b-1) &<& S_n(b)
            \text{ (an {\em up-change}) or} \\
            S_n(a) &>& 0 = S_n(a+1) = \ldots = S_n(b-1) &>& S_n(b)
            \text{ (a {\em down-change})},
          \end{array}\]
    \item an {\em extreme} if $n \geq 5$ and there are $a,b$ ($0 < a < b-1
          \leq n-3$) such that
          \[\begin{array}{rcccl}
            S_n(a) &<& S_n(a+1) = \ldots = S_n(b-1) &>& S_n(b)
            \text{ (a {\em maximum}) or} \\
            S_n(a) &>& S_n(a+1) = \ldots = S_n(b-1) &<& S_n(b)
            \text{ (a {\em minimum})},
          \end{array}\]
    \item an {\em inflection} if $n \geq 6$ and there are $a,b$ ($1 \leq a <
          b-2 \leq n-4$) such that
          \[\begin{array}{rcccl}
            0 < S_{n-1}(a+1) &<&
                S_{n-1}(a+2) = \ldots = S_{n-1}(b-1) &>&
                S_{n-1}(b) > 0
            \text{ or} \\
            0 > S_{n-1}(a+1) &>&
                S_{n-1}(a+2) = \ldots = S_{n-1}(b-1) &<&
                S_{n-1}(b) < 0.
          \end{array}\]
  \end{itemize}
  The pair $(a,b)$ is the {\em locus} of the change/extreme/inflection. We
  also say that the change/extreme/inflection is {\em located} at $a$.
\end{defn}

\begin{obs}
  Recall that $S_{n-1}(c) = S_n(c+1) - S_n(c)$, so this value acts as a
  ``discrete derivative'' of the sequence $S_n$ in the definition of
  inflection. Notice that in our inflections the slope at the center is
  steeper than at the sides. A corollary of property~\ref{NoBathtubs} below
  is that no other inflection shape is necessary. Also, note that $S_n(n-1)$
  is not allowed to be part of a change/extreme/inflection.
\end{obs}

\begin{prop}
 \label{prop:Shape}
  For all $n \geq 6$ the sequence $S_n$ satisfies
  {\renewcommand{\labelenumi}{(Sh\Alph{enumi})}
  \renewcommand{\theenumi}{(Sh\Alph{enumi})}
  \begin{enumerate}
    \item \label{OneOfEach}
          There are exactly one sign change, one extreme, and one inflection.
    \item \label{NoBathtubs}
          A maximum must have positive value and a minimum must have negative
          value.
    \item \label{NoMesas}
          There is at most one $r$ such that $S_n(r) = 0$.
    \item \label{HemiboundingBand}
          At least one of the two values $\min_r S_n(r)$ and $\max_r S_n(r)$
          lies between $x a_n$ and $y a_n$.
    \item \label{SlopesAtEndpoints}
          If the extreme has locus $(a,b)$ and the inflection has locus
          $(c,d)$, then
          {\renewcommand{\labelenumii}{(\alph{enumii})}
           \renewcommand{\theenumii}{(\alph{enumii})}
          \begin{enumerate}
            \item \label{ShapeEa}
                  $
                    a \leq c \Longrightarrow
                    S_{n-1}(0) < S_{n-1}(1) < S_{n-1}(2) < 0
                  $,
            \item \label{ShapeEb}
                  $
                    a > c \Longrightarrow
                    0 > S_{n-1}(0) > S_{n-1}(1) > S_{n-1}(2)
                  $.
          \end{enumerate} }
  \end{enumerate} }
\end{prop}


\begin{figure}[h]\refstepcounter{figure}\addtocounter{figure}{-1}
 \label{fig:ShapeIllustration}
  \begin{center}
    \includegraphics{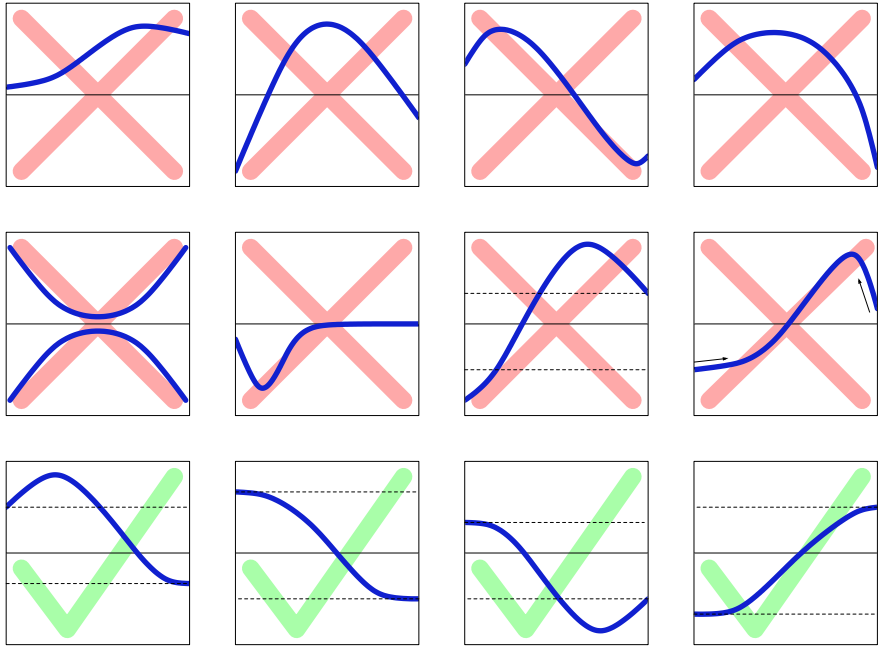}
    \caption{The first row of pictures illustrates property \ref{OneOfEach} of
             Proposition~\ref{prop:Shape}. (1): at least one zero. (2): at
             most one zero. (3): at most one extreme. (4): at least one
             inflection. The second row illustrates properties
             \ref{NoBathtubs}, \ref{NoMesas}, \ref{HemiboundingBand}, and
             \ref{SlopesAtEndpoints}.}
  \end{center}
\end{figure}

\begin{proof}[Proof of \ref{OneOfEach}]
A sign change in $S_n$ implies an extreme in $S_{n+1}$, which in turn
implies an inflection in $S_{n+2}$; therefore we only need to prove that $S_n$
has exactly one sign change.

For $n \geq 4$, $S_n(0) = y a_{n-1}$ and $S_n(n-2) = x a_{n-1}$. Since $x$ and
$y$ have opposite signs, the sequence $S_n$ has at least one sign change. If
there were more changes than one, there would be at least three because
$\text{sgn}\big( S_n(0) \big) \neq \text{sgn}\big( S_n(n-2) \big)$. Then $S_n$
has two extremes, and thus $S_{n-1}$ has two sign changes. But $S_4(2) =
S_4(3)$, so $S_4$ can have at most one sign change, so by induction, $S_n$ has
exactly one sign change.
%
%
%
\end{proof}

\begin{proof}[Proof of \ref{NoBathtubs}]
  Assume $S_n(n-1) > 0$ (the negative case is analogous). In particular,
  a maximum of $S_n$ must lie above $S_n(n-1)$ and thus be positive.
  Also, $S_n(0) = y a_{n-1} < 0$ because $S_n(n-1) = x a_{n-1}$. 
  {\renewcommand{\labelitemi}{$\circ$}
  \begin{itemize}
    \item If $S_n(1) < 0$, then a minimum must lie lower than
          $S_n(1)$ and thus be negative.
    \item If $S_n(1) > 0$, then there is an increase from $S_n(0)$,
          so $S_{n-1}(1) < 0$. Thus $S_{n-1}$ cannot have a down-change, and
          therefore $S_n$ cannot have a minimum. \qedhere
  \end{itemize} }
\end{proof}

\begin{proof}[Proof of \ref{NoMesas}]
  Since $S_n$ has at most one sign change, we need only discard the possibility
  of two (or more) consecutive zeros. Accordingly, assume that $S_n(r) =
  \ldots = S_n(r+k) = 0$ ($k \geq 1$) with $S_n(r-1) < 0$ and $S_n(r+k+1) \neq
  0$ (the case $S_n(r-1) > 0$ is analogous).

  Now, $S_n(r+k+1) < 0$ contradicts \ref{NoBathtubs}, so we can assume
  $S_n(r+k+1) > 0$. But then
  \[S_{n-1}(r) < 0 = S_{n-1}(r+1) = \ldots = S_{n-1}(r+k) > S_{n-1}(r+k+1),\]
  and this means that $S_{n-1}$ contradicts \ref{NoBathtubs}. \qedhere
\end{proof}

\begin{proof}[Proof of \ref{HemiboundingBand}]
  Assume $S_{n-1}(1) < 0$ so $S_n(1) > y a_{n-1}$ (the case $S_{n-1}(1) > 0$
  is analogous). Moreover, $S_{n-1}$ cannot have a down-change (because it
  starts with a negative value); therefore $S_n$ cannot have a minimum. Since
  $S_n$ starts above $y a_{n-1}$ and ends at $x a_{n-1}$, the conclusion
  follows. \qedhere
\end{proof}

\begin{proof}[Proof of \ref{SlopesAtEndpoints}]
  We can assume that the extreme of $S_n$ is a maximum (the minimum case is
  analogous). Then $S_{n-1}$ has an up-change at $(a+1,b)$, and an extreme at
  $(c+1,d)$.

  {\renewcommand{\labelenumi}{(\alph{enumi})}
  \begin{enumerate}
    \item If $a \leq c$, $S_{n-1}$ has an increase (namely the up-change)
          before its extreme. Hence the extreme is a maximum which,
          therefore, lies above $x a_{n-2}$. By \ref{HemiboundingBand},
          \[S_{n-1}(0) = y a_{n-2} \leq S_{n-1}(1) \leq S_{n-1}(2) < 0.\]
    \item If $a > c$, $S_{n-1}$ has an increase (namely the up-change) after
          its extreme. Hence the extreme is a minimum which, therefore, lies
          below $y a_{n-2}$. In particular, $S_{n-1}$ is decreasing until this
          minimum. We claim that it is decreasing starting at the auxiliary
          term, i.e., that $S_{n-1}(0) \geq S_{n-1}(1)$; otherwise, $S_{n-2}$
          has two sign changes. Then we have
          \[
            0 > y a_n = S_{n-1}(0) \geq S_{n-1}(1) \geq S_{n-1}(2).
            \qedhere
          \]
  \end{enumerate} }
\end{proof}

\subsection{{\boldmath$S_n$} becomes a step function}\label{subsect:Functional}
Formula~\eqref{eqn:PatternRecursion} induces a linear operator $\V{A}_n:
\R^{n-1} \longrightarrow \R^n$. Here we will embed $\V{A}_n$ as an integral
operator $A_n: \L \longrightarrow \L$, and find an operator $T$ which is the
limit of $\{ A_n \}$ in the operator norm. The goal will be to link the growth
of $\{a_n \}$ to the spectral properties of $T$. \\

For $n \geq 2$, the $r^{\rm th}$ entry of the column vector
\[
 \V{s}_n =
  \( [
    \begin{tabular}{c}
      $S_n(1) / (n-2)!$ \\
      $\vdots$ \\
      $S_n(n-1) / (n-2)!$
    \end{tabular}
  \) ] \in
  \R^{n-1}
\]
represents the average contribution to $a_n$ of patterns with last entry
$t_{n-1} = r$ (of which there are $(n-2)!$). With this notation,
Equation~\eqref{eqn:PatternRecursion} can be interpreted as a linear
transformation
\begin{equation}
 \label{eqn:AveragedPatternRecursion}
  \V{s}_{n+1} = (\V{A}_n \cdot \V{s}_n)/(n-1),
\end{equation}
where $\V{A}_n$ is the $n \times (n-1)$ matrix whose $(i,j)$-entry is $x$ if
$i \geq j$, and $y$ otherwise.

Let $E_n: \R^{n-1} \longrightarrow \L$ be the linear map that sends the
standard basis vector $\V{e}_j$ to the characteristic function of the interval
$\( [ \frac{j-1}{n-1}, \frac{j}{n-1} \) )$. The vector $\V{s}_n$ maps to
the step function $s_n = E_n(\V{s}_n)$, such that $s_n (u) = S_n(j)/(n-2)!$
whenever $u \in \( [\frac{j-1}{n-1}, \frac{j}{n-1} \) )$. The maps $\{ E_n \}$
embed the linear operators $\V{A}_n: \R^{n-1} \longrightarrow \R^n$ into
linear operators $A_n: \L \longrightarrow \L$. In particular,
Equation~\eqref{eqn:AveragedPatternRecursion} takes the form
\begin{equation}
 \label{eqn:IntegralPatternRecursion}
  s_{n+1}(u) =
  \( [ A_n (s_n) \) ](u) =
  \integral{0}{1}{\alpha_n(u,v) \cdot s_n(v)}{v},
\end{equation}
where the kernel $\alpha_n$ is a piecewise constant function whose value at
$(u,v) \in \( [ \frac{i-1}{n}, \frac{i}{n} \) ) \cross \( [ \frac{j-1}{n-1},
\frac{j}{n-1} \) )$ is
\[
  \alpha_n (u,v) =
  \begin{cases}
    x & \text{ if } i \geq j \\
    y & \text{ otherwise}
  \end{cases}
  \qquad
  \big(\text{i.e., equal to } (\V{A}_n)_{ij} \big).
\]

\begin{obs}
  The factor $1/(n-1)$ in \eqref{eqn:AveragedPatternRecursion} is hidden as a
  normalization factor in \eqref{eqn:IntegralPatternRecursion}. Indeed, when
  $u$ lies in the interval $\( [ \frac{i-1}{n}, \frac{i}{n} \) )$,
  \begin{align*}
    \( [ A_n (s_n) \) ](u) =
    & \integral{0}{1}{\alpha_n(u,v) \cdot E_n(\V{s}_n)(v)}{v} = \\
    & \sum_{j=1}^{n-1} \integral{\( [ \frac{j-1}{n-1}, \frac{j}{n-1} \) )}{}
        {(\V{A}_n)_{ij} \frac{S_n(j)}{(n-2)!}}{v} = \\
    &  \sum_{j=1}^{n-1} (\V{A}_n)_{ij} \frac{S_n(j)}{(n-2)!} \cdot
         \frac{1}{n-1},
  \end{align*}
  which is the $i^{\rm th}$ entry of $\V{s}_{n+1}$. In particular, (compare
  equation~\eqref{eqn:Recover-a_n}):
  \[
    a_n =
    (n-1)! \integral{0}{1}{s_n(v)}{v}.
  \]
\end{obs}

To prove the theorem we need to show that the rate of exponential decay of the
integrals $\integral{0}{1}{s_n(v)}{v}$ is bounded from below. The bound will
be dictated by the largest eigenvalue $\l$ of the limit operator of $A_n$.

\subsection{The limit operator {\boldmath$T$}}
Here we define the limit operator $T$ of the sequence $\{ A_n \}$, and
establish some of its basic properties.

Let $T: \L \longrightarrow \L$ by
\[
  (Tf)(u) =
  \integral{0}{1}{\kappa(u,v) \cdot f(v)}{v},
\]
with kernel
\[
  \kappa(u,v) =
  \begin{cases}
    x & \text{ if } u \geq v \\
    y & \text{ otherwise}.
  \end{cases}
\]

\begin{lemma2}
 \label{lemma:OperatorNormEstimate}
  The operator $T$ is the limit of $\{ A_n \}$ in the operator norm:
  \[\( \| T - A_n \) \| \leq \frac{1}{\sqrt{n}}.\]
\end{lemma2}

\begin{figure}[h]\refstepcounter{figure}\addtocounter{figure}{-1}
 \label{fig:Kernels}
  \begin{center}
    \includegraphics{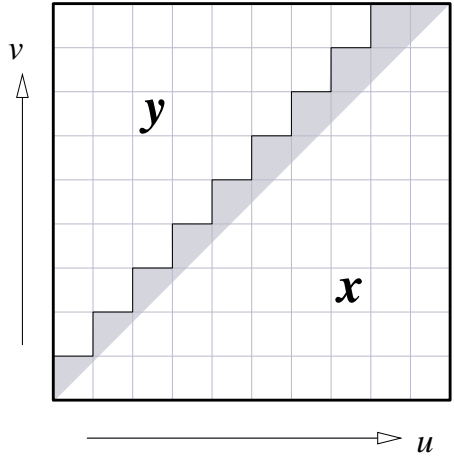}
    \caption{The kernel $\alpha_{10}$ (the shaded region is
             $\Omega_{10}$). Note that the unit square is divided into
             rectangles of size $\frac{1}{10}$ by $\frac{1}{9}$.}
  \end{center}
\end{figure}

\begin{proof}
  The kernel of $T - A_n$ is the function $\kappa - \alpha_n$. Since $y$
  is just a shorthand for $1+x$, we see that $\kappa - \alpha_n$ is the
  characteristic function of the staircase region $\Omega_n$ in the unit
  square, consisting of the upper triangle $\{ 0 \leq u,v \leq 1\; |\; u < v
  \}$, minus those rectangles $\( [ \frac{i-1}{n}, \frac{i}{n} \) ) \times \(
  [ \frac{j-1}{n-1}, \frac{j}{n-1} \) )$ such that $i < j$. Then the Lebesgue
  measure of $\Omega_n$ is
  \[
    \mu(\Omega_n) =
    \frac{1}{2} - \frac{(n-2)(n-1)}{2} \cdot \frac{1}{n(n-1)} =
    \frac{1}{n}.
  \]

  It follows that
  \[
    \( \| T - A_n \) \|_{L^2} \leq
    \big( \mu(\Omega_n) \big)^{1/2} =
    \frac{1}{\sqrt{n}}.
    \qedhere
  \]
\end{proof}

%

\subsection{Eigenfunctions of {\boldmath$T$}}
\label{subsect:Eigenfunctions}
The operator $T$ can be expressed as follows:
\begin{align*}
  (Tf)(u) =&
  \integral{0}{1}{\kappa(u,v) \cdot f(v)}{v} = \\
  & x \cdot \integral{0}{u}{f(v)}{v} + y \cdot \integral{u}{1}{f(v)}{v} = \\
  & -F(u) + y \cdot F(1) - x \cdot F(0),
\end{align*}
where $F$ is any primitive of $f$. To find the eigenvalues of $T$, set
\begin{align}
 \label{eqn:Eigenequation}
  (\l f)(u) =&
  (Tf)(u) = \\
  & -F(u) + y \cdot F(1) - x \cdot F(0), \notag
\end{align}
and differentiate to obtain the ODE
\[f'(u) = \frac{-f(u)}{\l},\]
with general solution
\[
  f(u) = C {\rm e}^{-u/\l}.
\]

A primitive of $f$ is $F(u) = -\l C {\rm e}^{-u/\l}$, so
substituting in \eqref{eqn:Eigenequation} gives
\[
  \l \big( C {\rm e}^{-1/\l} \big) =
  \l C {\rm e}^{-u/\l} -
    \l C \( ( y \cdot {\rm e}^{-1/\l} - x \) ).
\]

Thus, $\l$ is an eigenvalue if and only if
\[
  {\rm e}^{-1/\l} = \frac{x}{y}.
\]
Note that $\frac{x}{y} < 0$ exactly when $x \in (-1,0)$. Then we can write the
eigenvalues as
\[
  \l_m =
  \frac{-1}{\log \frac{x}{y} + 2 m \pi {\rm i}} =
  \frac{-1}{\log \big| \frac{x}{y} \big| + (2 m + 1) \pi {\rm i}}
  \text{ for all }m \in \Z,
\]
and the eigenfunction corresponding to $\l_m$ is
\[
  f_m(u) =
  \big| \tfrac{x}{y} \big|^u {\rm e}^{(2 m + 1) \pi {\rm i} u}.
\]

\begin{defn}
  For ease of notation, we write the absolute values of the two largest
  eigenvalues as $\l := |\l_{-1}| = |\l_0|$ and $\mu := |\l_{-2}| = |\l_1|$.
\end{defn}

\begin{lemma2}
 \label{lemma:OrthonormalBasis}
  The family of functions $\{ f_m  \}_{m \in \Z}$  forms a basis of $\L$.
\end{lemma2}

\begin{proof}
  Let $f$ be an arbitrary function in $\L$. Since $x \in (-1,0)$, the function
  $g(u) = \big| \frac{y}{x} \big|^u {\rm e}^{-\pi {\rm i} u}$ is continuous,
  so $fg \in \L$. After rescaling the standard basis of $L^2[-\pi, \pi]$, we
  obtain the representation
  \[
    (fg) (u) =
    \sum_{m \in \Z} c_m \cdot {\rm e}^{2 m \pi {\rm i} u},
  \]
  which implies
  \[
    f(u) =
    \sum_{m \in \Z} c_m \cdot
      \big| \tfrac{x}{y} \big|^u {\rm e}^{(2 m + 1) \pi {\rm i} u} =
    \sum_{m \in \Z} c_m \cdot f_m(u).
  \]
  The reverse argument shows that $\{ f_m \}$ are linearly independent.
\end{proof}

In order to turn $\{ f_m \}$ into an orthonormal basis, we introduce the {\em
weighted} inner product
\begin{equation}
 \label{eqn:InnerProduct}
  \langle f,g \rangle =
  \integral{0}{1}
    {\big| \tfrac{x}{y} \big|^{-2v} f(v) \overline{g}(v)}{v}.
\end{equation}

Note that for $m \geq 0$ the pair of functions $f_{-(m+1)}, f_m$ are complex
conjugate and their eigenvalues have the same magnitude. As a consequence, a
convenient basis for the subspace $\LR \subset \L$ of {\em real-valued}
functions is
\[
  \( \{
     \big| \tfrac{x}{y} \big|^u \cos \big( (2m+1) \pi u \big),
     \big| \tfrac{x}{y} \big|^u \sin \big( (2m+1) \pi u \big)
  \) \}_{m \geq 0}.
\]

\section{The functional approach}\label{sect:FunctionalApproach}
In this section we establish the lower bound on the exponential rate of decay
of the sequence $a_n/(n-2)!$ The long proof of
Lemma~\ref{lemma:Norms1AndInfty} interferes with the flow of logic, and is
consequently deferred to the next section. \\

\begin{defn}
  The eigenfunctions $f_{-1}$ and $f_0$ with largest eigenvalue $\l$ span a
  complex two-dimensional subspace of $\L$. Let $E \subset \LR$ denote the
  real slice of this subspace generated by $\( \{ \big| \tfrac{x}{y} \big|^u
  \cos \big( \pi u \big),\, \big| \tfrac{x}{y} \big|^u \sin \big( \pi u \big)
  \) \}$. The space $E^{\perp}$ spanned by all other eigenfunctions is
  orthogonal to $E$, so that (by Lemma~\ref{lemma:OrthonormalBasis}) $\LR = E
  \oplus E^{\perp}$. The projections onto $E$ and $E^{\perp}$ are denoted $P$
  and $P^{\perp}$ respectively. By Parseval's Theorem we can define the angle
  $\t_n$ by any of the three equivalent formulas
  \[
    \sin\t_n := \frac{\| P^{\perp} s_n \|_2}{\| s_n \|_2},\quad
    \cos\t_n := \frac{\| P s_n \|_2}{\| s_n \|_2},\quad
    \tan\t_n := \frac{\| P^{\perp} s_n \|_2}{\| P s_n \|_2}.
  \]
\end{defn}
Intuitively, the closer $\t_n$ is to $0$, the better $s_n$ resembles a
function in $E$. \\

Now we can describe the strategy of the proof:
{\renewcommand{\theenumi}{\arabic{enumi}}
\renewcommand{\labelenumi}{\underline{Step \arabic{enumi}}:}
\begin{enumerate}
  \item \label{ThetaBoundedAway}
        We use the shape properties of the sequence $S_n$ to show that the
        angles $\t_n$ are bounded away from $\pi/2$.
  \item \label{ThetaGoesTo0}
        The sequence $\{ \t_n \}$ converges to $0$, so the functions $s_n$
        become progressively sinusoidal.
  \item \label{ExponentialDecrease} There is a sequence of indices $\{ n_k \}$
        such that $\{ |a_{n_k}| \}$ is comparable to $\{ \|s_{n_k}\|_2
        \}$. Meanwhile, $\|s_n\|_2 \geq (\l-\e)^n$ for arbitrarily small $\e$,
        and the result will follow.
\end{enumerate}}

\subsection{Step \ref{ThetaBoundedAway} {\boldmath$(\t_n \leq \Theta <
            \pi/2)$}}
Fix $n$ and consider the locus $(a,b)$ of the sign change of $S_n$. The value
$z_n := a/(n-1) \in [0,1)$ is such that $s_n(u) \cdot \sin(\pi (u - z_n))$
never changes sign. We will show there is a $K>0$ such that for all $n$,
\begin{equation}
 \label{eqn:ProjectionAndNorm2}
  \big|
      \langle s_n(u), \big| \tfrac{x}{y} \big|^u \sin\pi(u - z_n \rangle
  \big| >
  K \|s_n\|_2.
\end{equation}
The projection $P s_n$ is larger than $|\langle s_n(u), \big|
\tfrac{x}{y} \big|^u \sin\pi(u - z_n \rangle|$, and thus the angle $\t_n$ is
bounded away from $\pi/2$ by
\[\t_n < \arccos K =: \Theta < \pi/2.\]

The proof of \eqref{eqn:ProjectionAndNorm2} follows from
Corollary~\ref{corol:Norms1And2} and
Lemma~\ref{lemma:PseudoProjectionAndNorm1} below, using $K = CC'$.

\begin{lemma2}
 \label{lemma:Norms1AndInfty}
  The sequences $\{ \|s_n\|_1 \}$ and $\{ \|s_n\|_{\infty} \}$ are comparable
  in the sense that there is a constant $C$ such that for all $n$,
  \[\|s_n\|_{\infty} \geq \|s_n\|_1 \geq C \|s_n\|_{\infty}.\]
\end{lemma2}

This is the main technical lemma, and its proof is deferred to
Section~\ref{sect:Proof-Norms1AndInfty}.

\begin{corol}
 \label{corol:Norms1And2}
  The sequences $\{ \|s_n\|_1 \}$ and $\{ \|s_n\|_2 \}$ are comparable in the
  sense that there is a constant $C$ such that for all $n$,
  \[\|s_n\|_2 \geq \|s_n\|_1 \geq C \|s_n\|_2.\]
\end{corol}

\begin{proof}
  The left side is the Cauchy-Schwarz inequality. On the right we have
  \[
    \|s_n\|_1^2 \geq
    C^2 \|s_n\|_{\infty}^2 \geq
    C^2 \integral{0}{1}{s_n^2(v)}{v} =
    C^2 \|s_n\|_2^2.
    \qedhere
  \]
\end{proof}

\begin{lemma2}
 \label{lemma:PseudoProjectionAndNorm1}
  Let $z_n$ be defined as above. Then there is a constant $C'$ such that
\[
  \big|
       \langle s_n(u), \big| \tfrac{x}{y} \big|^u \sin\pi(u - z_n) \rangle
  \big| >
  C' \|s_n\|_1.
\]
\end{lemma2}

\begin{proof}
  Given $z \in [0,1]$, the maximum of the function $|\tfrac{x}{y}|^{-u} \cdot
  |\sin\pi(u-z)|$ in a small interval $[z-\e,z+\e] \cap [0,1]$ is $M_{z,\e} >
  0$; a quantity that varies continuously. Fix $\e \equiv \e(z)$ so that the
  Lebesgue measure of the set
  \[
    L_z :=
    \{
      u \in [0,1] \text{ s.t. }
      |\tfrac{x}{y}|^{-u} \cdot |\sin\pi(u-z)| < M_{z,\e}
    \}
  \]
  is $C/2$, where $C$ is the constant of Lemma~\ref{lemma:Norms1AndInfty}.
  This is well defined because $|\tfrac{x}{y}|^{-u} \cdot |\sin\pi(u-z)|$ is
  nowhere constant, so the measure of $L_z$ varies continuously. Since $[0,1]$
  is compact, the lower bound $M := \inf_{z \in [0,1]} M_{z,\e}$ is positive.

  Recall that $z_n$ is chosen so that $s_n(u) \cdot \sin\pi(u-z_n)$ has
  constant sign. For simplicity, let us assume that this sign is
  positive. Then the weighted inner product $\big| \langle s_n(u), \big|
  \tfrac{x}{y} \big|^u\sin\pi(u-z_n) \rangle \big|$ equals (compare
  \eqref{eqn:InnerProduct}):
  \begin{gather*}
    \integral{0}{1}{|\tfrac{x}{y}|^{-v}s_n(v) \cdot \sin\pi(v-z_n)}{v} \geq
    \integral{L_{z_n}^{\complement}}{}
             {|\tfrac{x}{y}|^{-v}s_n(v) \cdot \sin\pi(v-z_n)}{v} > \\
    M \integral{L_{z_n}^{\complement}}{}{|s_n(v)|}{v} =
    M \( (
          \integral{0}{1}{|s_n(v)|}{v} - \integral{L_{z_n}}{}{|s_n(v)|}{v}
      \) ).
  \end{gather*}
  But
  \[\integral{L_{z_n}}{}{|s_n(v)|}{v} <
    \mu(L_{z_n}) \|s_n\|_{\infty} =
    \tfrac{C}{2} \|s_n\|_{\infty} <
    \tfrac{1}{2} \|s_n\|_{1}
  \]
  by Lemma~\ref{lemma:Norms1AndInfty}, so we get
  \[
    \big|
         \langle s_n(u), \big| \tfrac{x}{y} \big|^u \sin\pi(u-z_n) \rangle
    \big| >
    \tfrac{M}{2} \|s_n\|_1.
    \qedhere
  \]
\end{proof}

\subsection{Step \ref{ThetaGoesTo0}
            {\boldmath$(\t_n \rightarrow 0)$}}
We will show in Lemma~\ref{lemma:TanBounds} that when $n$ is large
enough, the sequence $\{ \tan \t_n \}$ enters a decreasing regime that makes
it eventually converge to 0. This establishes the desired result.
  
First we derive two versions of the basic estimate for $\tan\t_{n+1}$:
\[
  \tan \t_{n+1} =
  \frac{\|P^{\perp} A_n s_n\|_2}{\|P A_n s_n\|_2} =
  \frac{\|P^{\perp}[Ts_n + (A_n-T)]\|_2}{\|P[Ts_n + (A_n-T)]\|_2} \leq
  \frac{\|P^{\perp}Ts_n\|_2 + \|P^{\perp}(A_n-T)s_n\|_2}
       {\big| \|PTs_n\|_2 - \|P(A_n-T)s_n\|_2 \big|},
\]
and that $\theta_n < \Theta$ allows us to remove the absolute value in the
denominator by assuming $n$ is large enough. Since $T$ commutes with the
projections $P$ and $P^{\perp}$, and using
Lemma~\ref{lemma:OperatorNormEstimate},
\begin{align}
  \tan \t_{n+1} & \leq
  \frac{\mu\|P^{\perp}s_n\|_2 + \|s_n\|_2/\sqrt{n}}
       {\l\|Ps_n\|_2 - \|s_n\|_2/\sqrt{n}} = \notag \\
  &= \frac{\sqrt{n}\mu\sin\t_n + 1}
          {\sqrt{n}\l\cos\t_n - 1} \label{eqn:AbsoluteTanBound} \\
  &= \( (
         \frac{\mu + \frac{1}{\sqrt{n}\sin\t_n}}
              {\l - \frac{1}{\sqrt{n}\cos\t_n}}
     \) )\tan \t_n. \label{eqn:RelativeTanBound}
\end{align}

\begin{lemma2}
 \label{lemma:TanBounds}
  Let $n >\frac{9}{\cos^2\Theta (\l-\mu)^2}$. Then there are constants $0 < \e
  < 1$ and $R > 0$ such that
  {\renewcommand{\theenumi}{\alph{enumi})}
   \renewcommand{\labelenumi}{\alph{enumi})}
  \begin{enumerate}
    \item If $\sqrt{n}\sin\t_n \geq \frac{3}{\l-\mu}$, then
          \[\tan\t_{n+1} < (1-\e) \tan\t_n.\]
    \item If $\sqrt{n}\sin\t_n \leq \frac{3}{\l-\mu}$, then
          \[\tan\t_{n+1} < \frac{R}{\sqrt{n}}.\]
  \end{enumerate} }
\end{lemma2}
In other words, when $n$ is sufficiently large, each step in the sequence
$\{ \tan\t_n \}$ affords a definite relative decrease, or a slower but
absolute decrease.

\begin{proof}[Proof of Lemma~\ref{lemma:TanBounds}]
\ \\
  {\renewcommand{\theenumi}{\alph{enumi})}
   \renewcommand{\labelenumi}{\alph{enumi})}
  \begin{enumerate}
    \item The angle $\t_n$ is smaller than $\Theta$, so the initial assumption
          on $n$ implies
          \[
            \l - \frac{1}{\sqrt{n}\cos\t_n} >
            \l - \frac{1}{\sqrt{n}\cos\Theta} >
            \l - \frac{\l-\mu}{3}.
          \]
          The condition $\sqrt{n}\sin\t_n \geq \frac{3}{\l-\mu}$ is equivalent
          to
          \[\mu + \frac{1}{\sqrt{n}\sin\t_n} < \mu + \frac{\l-\mu}{3},\]
          so \eqref{eqn:RelativeTanBound} is smaller than
          \[
            \( ( \frac{\mu+\frac{\l-\mu}{3}}{\l-\frac{\l-\mu}{3}} \) )
            \tan\t_n =
            \( ( \frac{2\mu+\l}{2\l+\mu} \) ) \tan\t_n.
          \]
          Since $\frac{2\mu+\l}{2\l+\mu} < 1$, this case is proved.
    \item When $\sqrt{n}\sin\t_n \leq \frac{3}{\l-\mu}$, we also have
          $\sqrt{n}\cos\t_n \geq \sqrt{n - \frac{9}{(\l-\mu)^2}}$.
          Substituting in \eqref{eqn:AbsoluteTanBound} gives
          \[
            \tan\t_{n+1} <
            \frac{\mu\frac{3}{\l-\mu}+1}{\l\sqrt{n-\frac{9}{(\l-\mu)^2}}-1},
          \]
          and the result follows. \qedhere
  \end{enumerate} }
\end{proof}

\subsection{Step \ref{ExponentialDecrease}
            {\boldmath$(|\int s_{n_k}| >$ cst\,$(\l-\e)^{n_k})$}}
We are ready to prove that
\begin{equation}
 \label{eqn:Success}
  \(| \integral{0}{1}{s_n(v)}{v} \)| \geq (\l-\e)^n
\end{equation}
along a subsequence of indices. The immediate consequence is factorial growth
of $|a_n|$, since $a_n = (n-2)! \cdot \integral{0}{1}{s_n(v)}{v}$.
Equation~\eqref{eqn:Success} follows at once from propositions
\ref{prop:snAndNorm2} and \ref{prop:Norm2AndLambda}.
\begin{prop}
 \label{prop:snAndNorm2}
  There is an infinite integer sequence $n_1 < n_2 < \ldots$, and a constant
  $0 < W \leq 1$ such that for all $k$,
  \[
    \(| \integral{0}{1}{s_{n_k}(v)}{v} \)| \geq
    \tfrac{W}{\sqrt{17}} \| s_{n_k} \|_2.
  \]
\end{prop}

\begin{prop}
 \label{prop:Norm2AndLambda}
  For every $\e>0$ there are constants $G,N>0$ such that for $n>N$,
  \[\| s_n \|_2 > G (\l-\e)^n.\]
\end{prop}

The proof of Proposition~\ref{prop:snAndNorm2} uses the following auxiliary
result:

\begin{lemma2}
 \label{lemma:ProjectionAndNorm1}
  There is an infinite integer sequence $n_1 < n_2 < \ldots$ such that for
  all $k$,
  \begin{equation}
   \label{eqn:Ptriangle}
    \(| \integral{0}{1}{P s_{n_k}(v)}{v} \)| \geq
    \tfrac{1}{2} \integral{0}{1}{|P s_{n_k}(v)|}{v}.
  \end{equation}
\end{lemma2}

We prove Lemma~\ref{lemma:ProjectionAndNorm1} first, and then propositions
\ref{prop:snAndNorm2} and \ref{prop:Norm2AndLambda}.

\begin{proof}[Proof of Lemma~\ref{lemma:ProjectionAndNorm1}]
  First notice that $\integral{0}{1}{\sin(\pi(u+\omega))}{v} = \tfrac{2}{\pi}
  \cos \pi\omega$, so
  \begin{equation}
   \label{eqn:OutOfPhaseSine}
    \integral{0}{1}{\sin(\pi(u+\omega))}{v} \geq
    \tfrac{1}{\sqrt{2}} \integral{0}{1}{\sin(\pi u)}{v}
    \text{ if and only if }
    |\omega \imod{1}| \leq 1/4.
  \end{equation}

  The lemma will follow once we prove that if a large enough $n$ does not
  satisfy \eqref{eqn:Ptriangle}, then $n+1$ does. Accordingly, assume that
  \[
    \(| \integral{0}{1}{P s_n(v)}{v} \)| \leq
    \tfrac{1}{2}\integral{0}{1}{|P s_n(v)|}{v} <
    \tfrac{1}{\sqrt{2}} \integral{0}{1}{|P s_n(v)|}{v}.
  \]
  Since the function $P s_n$ is in $E$, it has the form $P s_n(u) = B \sin
  (\pi(u+\omega))$ for some constants $B$, $\omega$. According to
  \eqref{eqn:OutOfPhaseSine}, the assumption above means that $|\omega
  \imod{1}| \geq 1/4$. Now,
  \begin{equation}
   \label{eqn:Ps{n+1}Broken}
    \(| \integral{0}{1}{P s_{n+1}(v)}{v} \)| =
    \(|
       \integral{0}{1}{PT s_n(v)}{v} +
       \integral{0}{1}{P(A_n - T)s_n(v)}{v}
    \)|.
  \end{equation}
  We estimate both terms on the right side.  On one hand, since $P$
  commutes with $T$, and $|(\omega+1/2) \imod{1}| \leq 1/4$,
  \begin{gather*}
    \(| \integral{0}{1}{PT s_n(v)}{v} \)| =
    \(| \integral{0}{1}{\l B \sin(\pi(v+\omega+1/2))}{v} \)| \geq \\
    \tfrac{\l}{\sqrt{2}} \integral{0}{1}{|B \sin(\pi(u+\omega+1/2))|}{v} =
    \tfrac{\l}{\sqrt{2}} \integral{0}{1}{|B \sin(\pi(u+\omega))|}{v} =
    \tfrac{\l}{\sqrt{2}} \|P s_n\|_1.
  \end{gather*}

  On the other hand,
  \begin{gather*}
    \(| \integral{0}{1}{P(A_n - T)s_n(v)}{v} \)| \leq
    \integral{0}{1}{|P(A_n - T)s_n(v)|}{v} = \\
    \|P(A_n - T) s_n\|_1 \leq
    \|P(A_n - T) s_n\|_2 \leq
    \|(A_n - T) s_n\|_2 \leq
    \tfrac{1}{\sqrt{n}} \|s_n\|_2
  \end{gather*}
  by Lemma~\ref{lemma:OperatorNormEstimate}. Since $\t_n < \Theta < \pi/2$,
  the last quantity is smaller than
  \[
    \tfrac{1}{\cos\Theta\sqrt{n}} \|P s_n\|_2 \leq
    \tfrac{1}{\cos\Theta\sqrt{n}} \tfrac{2\sqrt{2}}{\pi} \|P s_n\|_1,
  \]
  where the last estimate comes from comparing the 1- and 2-norms of a sine
  function. Altogether, plugging both estimates in \eqref{eqn:Ps{n+1}Broken}
  gives
  \begin{equation}
   \label{eqn:LadyGaga}
    \(| \integral{0}{1}{P s_{n+1}(v)}{v} \)| \geq
    \big(
         \tfrac{\l}{\sqrt{2}} -
         \tfrac{2\sqrt{2}}{\pi\cos\Theta\sqrt{n}}
    \big) \|P s_n \|_1.
  \end{equation}
  It only rests to compare $\|P s_n\|_1$ with $\|P s_{n+1}\|_1 = \|PT s_n +
  P(A_n-T) s_n\|_1$. Since $P$ and $T$ commute, $\|PT s_n\|_1 = \l\|P
  s_n\|_1$. Also, $\|P(A_n-T) s_n\|_1 \leq
  \tfrac{2\sqrt{2}}{\pi\cos\Theta\sqrt{n}} \|P s_n\|_1$, as we saw above.
  This gives
  \[
    \|P s_n\|_1 \geq
    \big( \l + \tfrac{2\sqrt{2}}{\pi\cos\Theta\sqrt{n}} \big)^{-1}
    \|P s_{n+1}\|_1,
  \]
  which, plugged back in \eqref{eqn:LadyGaga} gives
  \[
    \(| \integral{0}{1}{P s_{n+1}(v)}{v} \)| \geq
    \( ( \frac{\tfrac{\l}{\sqrt{2}} - \tfrac{2\sqrt{2}}{\pi\cos\Theta\sqrt{n}}}
         {\l + \tfrac{2\sqrt{2}}{\pi\cos\Theta\sqrt{n}}} \) )
      \|P s_{n+1}\|_1.
  \]
  For large enough $n$ the last quantity is larger than $\tfrac{1}{2} \|P
  s_{n+1}\|_1$, and the result follows.
\end{proof}

With the above result we are ready to prove propositions \ref{prop:snAndNorm2}
and \ref{prop:Norm2AndLambda}, establishing \eqref{eqn:Success} and our main
result.

\begin{proof}[Proof of Proposition~\ref{prop:snAndNorm2}]
  Consider the sequence $\{ n_k \}$ from Lemma~\ref{lemma:ProjectionAndNorm1},
  truncated in the beginning so that the right hand expression in
  \begin{multline}
   \label{eqn:snAndNorm1}
    \(| \integral{0}{1}{s_{n_k}(v)}{v} \)| =
    \(|
       \integral{0}{1}{P s_{n_k}(v)}{v} +
       \integral{0}{1}{P^{\perp} s_{n_k}(v)}{v}
    \)| \geq \\
    \(| \integral{0}{1}{P s_{n_k}(v)}{v} \)| -
       \(| \integral{0}{1}{P^{\perp} s_{n_k}(v)}{v} \)|
  \end{multline}
  is positive. Let us evaluate both terms. By
  Lemma~\ref{lemma:ProjectionAndNorm1},
  \[
    \(| \integral{0}{1}{P s_{n_k}(v)}{v} \)| \geq
    \tfrac{1}{2} \|P s_{n_k} \|_1 \geq
    \tfrac{W}{2} \|P s_{n_k} \|_2,
  \]
  where $0 < W \leq 1$ is a lower bound on the quotient of the 1- and 2-norms
  of $\(| \tfrac{x}{y} \)|^u \sin\pi(u+\phi)$ (here $0 \leq \phi \leq 1$ is an
  arbitrary phase shift). On the other hand, the triangle and Cauchy-Schwarz
  inequalities give
  \[
    \(| \integral{0}{1}{P^{\perp} s_{n_k}(v)}{v} \)| \leq
    \|P^{\perp} s_{n_k} \|_2.
  \]
  If $n_k$ is sufficiently large, then $\tan \theta_{n_k} < W/4 \leq 1/4$,
  where $W$ is as before, and we get
  \[
    \|P^{\perp} s_{n_k} \|_2 =
    \tan \theta_{n_k} \|P s_{n_k} \|_2 \leq
    \tfrac{W}{4} \|P s_{n_k} \|_2.
  \]
  Plugging these estimates back in \eqref{eqn:snAndNorm1} gives
  \[
    \(| \integral{0}{1}{s_{n_k}(v)}{v} \)| \geq
    \tfrac{W}{2} \|P s_{n_k}\|_2 - \tfrac{W}{4} \|P s_{n_k}\|_2 =
    \tfrac{W}{4} \|P s_{n_k}\|_2 =
    \tfrac{W}{4} \cos \theta_{n_k} \|s_{n_k}\|_2 >
    \tfrac{W}{\sqrt{17}} \|s_{n_k}\|_2,
  \]
  since $\tan \theta_{n_k} < 1/4$.
\end{proof}

\begin{proof}[Proof of Proposition~\ref{prop:Norm2AndLambda}]
  For large enough $n$, the right hand side of
  \[
    \|s_{n+1}\|_2 =
    \|T s_n + (A_n-T) s_n\|_2 \geq
    \|T s_n\|_2 - \|(A_n-T) s_n\|_2
  \]
  is positive. In fact, the two terms on the right have the bounds
  \[
    \|T s_n\|_2 \geq
    \|TP s_n\|_2 =
    \l \cos \theta_n \|s_n\|_2,
  \]
  and
  \[
    \|(A_n-T) s_n\|_2 \leq
    \tfrac{1}{\sqrt{n}}\|s_n\|_2
  \]
  by Lemma~\ref{lemma:OperatorNormEstimate}, so
  \[
    \|s_{n+1}\|_2 \geq
    \big( \l\cos\theta_n - \tfrac{1}{\sqrt{n}} \big) \|s_n\|_2.
  \]
  Since $\t_n < \Theta$, the result follows.
\end{proof}

\section{Proof of Lemma~\ref{lemma:Norms1AndInfty}}
\label{sect:Proof-Norms1AndInfty}
The inequality $\|s_n\|_1 \leq \|s_n\|_{\infty}$ is trivial, but the opposite
direction requires estimates based on the shape of the sequences
$S_n$. Because of the rescaling
\[\frac{S_n(j)}{(n-2)!} = s_n \( (\tfrac{j-1/2}{n-1} \) ),\]
all statements about the shape of the sequence $S_n$ can be interpreted as
applying to the function $s_n$. The idea of the proof is as follows: Our
definition of inflection yields a natural concept of concavity for step
functions. Within each interval of concavity we find suitable linear functions
that bound $|s_n|$ from below as illustrated in
Figure~\ref{fig:n1Estimate}. Then we show that these bounds are comparable to
the maximum $\|s_n\|_{\infty}$.

\medskip\noindent
{\bf Basic assumption:} Let the sign change of $S_n$ be located at $Z$, the
extreme at $E$, and the inflection at $I$. We will assume that the extreme is
a minimum and that $E < I$, the other cases being analogous by symmetry. Note
that $S_{n-1}$ has a minimum at $I$ and a down-change at $E$. In particular,
$S_{n-1}$ is negative to the right of $E$, so $S_{n-1}(0) = \frac{y}{-x}
S_{n-1}(n-2) > 0$. Then $S_{n-1}(1) < S_{n-1}(0)$ by
property~\ref{HemiboundingBand}, and $S_{n-1}(1) > 0$ because it is to the
left of the down-change.

\begin{defn}
  To avoid carrying factors of $\frac{1}{n-1}$ we follow the convention that
  indices from $1$ to $n-1$ are represented by capital letters, and their
  counterparts in the interval $[0,1]$ by the corresponding lowercase
  letter. In particular, we let $e := \frac{E-1/2}{n-1}$, $i :=
  \frac{I-1/2}{n-1}$, and $z := \frac{z-1/2}{n-1}$.
\end{defn}

\begin{defn}
  For given $n$, and integers $1 \leq A,B \leq n-1$, let $a :=
  \frac{A-1/2}{n-1}$, $b := \frac{B-1/2}{n-1}$. We denote by
  $\Lambda_n(a,w_1,b,w_2)$ the linear function whose graph is the straight
  line from $\big( a,w_1 \big)$ to $\big( b,w_2 \big)$.  Also, let $\mu(a,b)$
  be the length $b-a$ of the interval $[a,b]$.
\end{defn}
For our purposes, $w_j$ will always be 0 or $s_n(a)$ for some $a \in [0,1]$.
Consequently, if $s_n$ is positive, and $s_{n-1}$ is increasing (so $s_n$ is
``concave'') from $a$ to $b$, the function $\Lambda_n(A,s_n(a),B,0)$ is also
positive in $(a, b)$. Moreover, its integral gives a lower bound for $\int
|s_n|$ on every intermediate interval where $s_n$ is constant. \\

\begin{figure}[h]\refstepcounter{figure}\addtocounter{figure}{-1}
 \label{fig:n1Estimate}
  \begin{center}
    \includegraphics{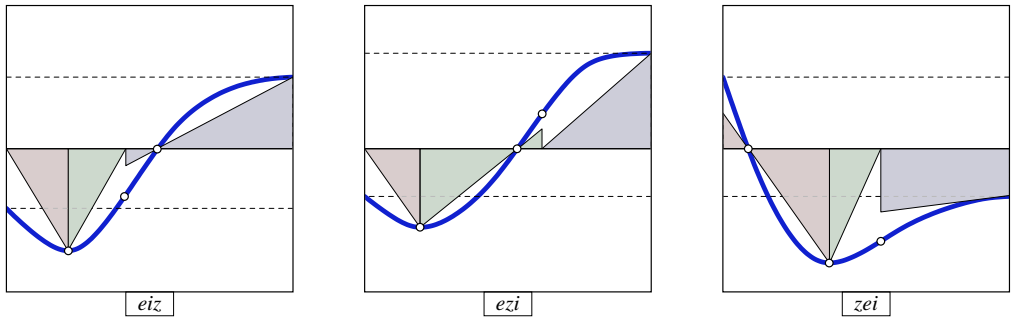}
    \caption{These continuous functions are caricatures of the step function
             $s_n$ in the different situations of cases 1, 2, 3. The 1-norm is
             bounded from below by the shaded areas under the linear
             functions. We show that these areas are comparable to the maximum
             $\|s_n\|_{\infty}$.}
  \end{center}
\end{figure}

The function $s_n$ can adopt one of three forms depending on the order of $e$,
$i$, and $z$. In each case we split $[0,1]$ into the same three intervals
$[0,e]$, $[e,i]$, and $[i,1]$, and describe linear functions on these
intervals that bound $|s_n|$ from below (compare Figure~\ref{fig:n1Estimate}).

{\renewcommand{\theenumi}{\arabic{enumi}}
\renewcommand{\labelenumi}{\underline{Case \arabic{enumi}}}
\begin{enumerate}
\item \label{eiz}
      $(e<i<z)$: The linear functions are
      \begin{enumerate}
        \item \label{eiz-0e}
              On $[0,e]$: $\Lambda(0,0,e,s_n(e))$. The area of the triangle is
              \[\boxed{\tfrac{1}{2} \cdot |s_n(e)| \cdot \mu(0,e)}\]
        \item \label{eiz-ei}
              On $[e,i]$: $\Lambda(e,s_n(e),i,0)$. The area of the triangle is
              \[\boxed{\tfrac{1}{2} \cdot |s_n(e)| \cdot \mu(e,i)}\]
        \item \label{eiz-i1}
              On $[i,1]$: $\Lambda(z,0,1,s_n(1))$. The slope is
              $\frac{s_n(1)}{\mu(z,1)}$ so the areas of the two triangles are
              \[
                \frac{|s_n(1)|}{\mu(z,1)} \cdot
                  \( ( \frac{\mu(z,1)^2}{2} + \frac{\mu(i,z)^2}{2} \) ) \geq
                \frac{|s_n(1)|}{\mu(z,1)} \cdot \frac{\mu(i,1)^2}{4} \geq
                \boxed{\tfrac{1}{4} \cdot|s_n(1)| \cdot \mu(i,1)}
              \]
      \end{enumerate}
\item \label{ezi}
      $(e<z<i)$: The linear functions are
      {\renewcommand{\theenumii}{(\alph{enumii})}
      \renewcommand{\labelenumii}{(\alph{enumii})}
      \begin{enumerate}
        \item \label{ezi-0e}
              On $[0,e]$: $\Lambda(0,0,e,s_n(e))$. The area of the triangle is
              \[\boxed{\tfrac{1}{2} \cdot |s_n(e)| \cdot \mu(0,e)}\]
        \item \label{ezi-ei}
              On $[e,i]$: $\Lambda(e,s_n(e),z,0)$. The slope is
              $\frac{|s_n(e)|}{\mu(e,z)}$. Since $a^2+b^2 \geq (a+b)^2/2$, the
              areas of the two triangles are
              \[
                \frac{|s_n(e)|}{\mu(e,z)} \cdot
                  \( ( \frac{\mu(e,z)^2}{2} + \frac{\mu(z,i)^2}{2} \) ) \geq
                \frac{|s_n(e)|}{\mu(e,z)} \cdot \frac{\mu(e,i)^2}{4} \geq
                \boxed{\tfrac{1}{4} \cdot|s_n(e)| \cdot \mu(e,i)}
              \]
        \item \label{ezi-i1}
              On $[i,1]$: $\Lambda(i,0,1,s_n(1))$. The area of the triangle is
              \[\boxed{\tfrac{1}{2} \cdot |s_n(1)| \cdot \mu(i,1)}\]
      \end{enumerate}
\item \label{zei}
      $(z<e<i)$: The linear functions are
      \begin{enumerate}
        \item \label{zei-0e}
              On $[0,e]$: $\Lambda(z,0,e,s_n(e))$. The slope is
              $\frac{s_n(e)}{\mu(z,e)}$ so the areas of the two triangles are
              \[
                \frac{|s_n(e)|}{\mu(z,e)} \cdot
                  \( ( \frac{\mu(0,z)^2}{2} + \frac{\mu(z,e)^2}{2} \) ) \geq
                \frac{|s_n(e)|}{\mu(z,e)} \cdot \frac{\mu(0,e)^2}{4} \geq
                \boxed{\tfrac{1}{4} \cdot|s_n(e)| \cdot \mu(0,e)}
              \]
        \item \label{zei-ei}
              On $[e,i]$: $\Lambda(e,s_n(e),i,0)$. The area of the triangle is
              \[\boxed{\tfrac{1}{2} \cdot |s_n(e)| \cdot \mu(e,i)}\]
        \item \label{zei-i1}
              On $[i,1]$: $\Lambda(a,s_n(a),b,s_n(b))$, where $A =
              n-3$, $B = n-2$, and  $a = \frac{A-1/2}{n-1}$, $b =
              \frac{B-1/2}{n-1}$. The region bounded by this $\Lambda$
              consists of a rectangle of base $\mu(i,1)$ and height $s_n(1)$,
              plus a triangle of base $\mu(i,1)$ and slope $s_{n-1}(b)$. The
              area of $s_n$ is at least
              \[
                \boxed{
                  |s_n(1)| \cdot \mu(i,1) +
                  |s_{n-1}(b)| \cdot \tfrac{\mu(i,1)^2}{2}.
                }
              \]
      \end{enumerate} }
\end{enumerate} }

In each of the three cases, $\|s_n\|_1$ is bounded by a sum of three
estimates. A trivial weakening of these expressions allows us to consolidate
cases \ref{eiz} and \ref{ezi} into one:
\begin{itemize}
  \item [\underline{\ref{eiz}{\tiny\&}\ref{ezi}}:]
        If $(e<z)$, then $\|s_n\|_1 \leq \tfrac{1}{4} \Big( |s_n(e)| \cdot
        \mu(0,i) + |s_n(1)| \cdot \mu(i,1) \Big)$.
  \item [\underline{\ref{zei}}:]
        If $(z<e)$, then $\|s_n\|_1 \leq \tfrac{1}{4} \Big( |s_n(e)| \cdot
        \mu(0,i) + |s_n(1)| \cdot \mu(i,1) + |s_{n-1}(b)| \cdot
        \mu(i,1)^2 \Big)$.
\end{itemize}

Note that $|s_n(e)| \geq |s_n(0)| = \frac{|x|}{y}|s_n(1)|$. If the interval
$[0,i]$ has definite size, say $\mu(0,i) > \frac{|x|}{6}$, then we can neglect
the portion of the bounds that contains $\mu(i,1)$ and see that
\[
  \|s_n\|_1 \geq
  \tfrac{1}{4}\tfrac{|x|}{6}|s_n(e)| \geq
  \tfrac{x^2}{24}|s_n(1)|
\]
in all three cases. Since $\|s_n\|_{\infty}$ is either $|s_n(e)|$ or
$|s_n(1)|$, we find
\[\|s_n\|_1 \geq \tfrac{x^2}{24y}\|s_n\|_{\infty}.\]

To finish the proof we have to consider what happens when $\mu(0,i) <
\frac{|x|}{6}$. In this situation we neglect the portion of the bounds that
contains $\mu(0,i)$, and show that both $|s_n(1)|$ and $|s_{n-1}(b)|$ have
lower bounds \eqref{eqn:Final1}, \eqref{eqn:Final2}, \eqref{eqn:Final3} of the
form constant times $|s_n(e)|$. Just as above, this implies a lower bound for
$\|s_n\|_1$ in terms of $\|s_n\|_{\infty}$, and we are done.

As a note of caution, note that for this final step we revert to the language
of sequences $S_n$. Thus, instead of seeking bounds for $s_n(1), s_{n-1}(b)$
in terms of $s_n(e)$, we get equivalent bounds for $S_n(n-2), S_{n-1}(n-3)$ in
terms of $S_n(E)$.

\begin{lemma2}
 \label{lemma:Cheesecake}
  If $I < \frac{|x|n}{6}$ then $|S_{n-1}(1)| \leq \frac{1}{2} |S_n(0)|$.
\end{lemma2}

From Lemma~\ref{lemma:Cheesecake} follows $|S_n(0)| \leq
\tfrac{2}{3}|S_n(1)|$ (see \eqref{eqn:BasicSequenceRelation}). Hence, if
$|S_n(1)| \geq \frac{|S_n(E)|}{2}$, we get
\begin{equation}
 \label{eqn:Final1}
  S_n(n-1) =
  \frac{y}{|x|} |S_n(0)| \geq
  \frac{2}{3}\frac{y}{|x|} |S_n(1)| \geq
  \frac{y}{3|x|} |S_n(E)|.
\end{equation}
It rests only to consider what happens when
\[|S_n(1)| \leq \frac{|S_n(E)|}{2}.\]

If this is the case, property~\ref{SlopesAtEndpoints} gives the following bound
for cases \ref{eiz} and \ref{ezi}:
\begin{equation}
 \label{eqn:Final2}
  \frac{|S_n(1)|}{\mu((1-\eta),1)} \geq
  \frac{y}{|x|} \frac{|S_n(E) - S_n(1)|}{\mu(0,E)} \geq
  \frac{y}{|x|} |S_n(E) - S_n(1)|,
\end{equation}
where $\eta = \max \{ I,Z \}$.

In case~\ref{zei}, property~\ref{SlopesAtEndpoints} gives a different bound:
\begin{equation}
 \label{eqn:Final3}
  |S_{n-1}(n-3)| \geq \frac{y}{|x|} \frac{|S_n(E)|}{\mu(Z,I)}.
\end{equation}

\begin{proof}[Proof of Lemma~\ref{lemma:Cheesecake}]
  Throughout this section the basic assumption has been that $S_n(E) < 0$
  and $E < I$, so that $S_{n-1}$ has a minimum at $I$ and a down-change at
  $E$. Recall that this implies $S_{n-1}$ is negative to the right of $E$,
  and $0 < S_{n-1}(1) < S_{n-1}(0)$. \\

  First we derive Inequality~\eqref{eqn:Ineq}. Since $y-x = 1$ and $n/2 \leq
  n-1$,
  \[
    I <
    \frac{-xn}{6} =
    \frac{-x (n/2)}{2+y-x} \leq
    \frac{-x (n-1)}{2+y-x},
  \]
  so
  \[\frac{-x(n-1-I)}{(2+y)I} \geq 1.\]
  The left expression increases if $I$ is replaced below by $E$:
  \[\frac{-x(n-1-I)}{(2+y)E} \geq 1,\]
  giving
  \begin{equation}
   \label{eqn:yE+mx<0}
    -2E \geq yE + x(n-1-I).
  \end{equation}
  Now, the right expression is negative because it is smaller than
  $yI + x(n-1-I)
  = I + x(n-1) < I + \frac{xn}{6} < 0$, so dividing in \eqref{eqn:yE+mx<0} and
  multiplying by $y$ gives
  \begin{equation}
   \label{eqn:Ineq}
    \frac{-yE}{yE + x(n-1-I)} \leq
    \frac{y}{2}.
  \end{equation}

  \medskip
  Using inequality~\eqref{eqn:Ineq} we prove Lemma~\ref{lemma:Cheesecake} as
  follows. The absolute value of $S_{n-1}$ is decreasing from 1 to $E$, and
  from $I$ to $n-2$. In the second of these spans the average exceeds the
  rightmost term $S_{n-1}(n-2)$; thus,
  \begin{gather}
   \label{eqn:A}
    \(| \sum_{j=1}^E S_{n-1}(j) \)| <
    E |S_{n-1}(1)| <
    E |S_{n-1}(0)| =
    E \(| \tfrac{y}{x} S_{n-1}(n-2) \)| < \\ \notag
    \tfrac{yE}{|x|m} \(| \sum_{j=I}^{n-2} S_{n-1}(j) \)| <
    \tfrac{yE}{|x|m} \(| \sum_{j=E+1}^{n-2} S_{n-1}(j) \)|.
  \end{gather}
  where $m := n-1-I$.

  Substitute this estimate in
  \begin{gather}
   \label{eqn:B}
    \(| \sum_{j=1}^{n-2} S_{n-1}(j) \)| =
    \(| \sum_{j=E+1}^{n-2} S_{n-1}(j) \)| -
      \(| \sum_{j=1}^{E} S_{n-1}(j) \)| \geq \\
    \( ( 1 - \tfrac{yE}{|x|m} \) ) \(| \sum_{j=E+1}^{n-2} S_{n-1}(j) \)|.
    \notag
  \end{gather}

  Now, using \eqref{eqn:A}, \eqref{eqn:B}, and \eqref{eqn:Ineq},
  \begin{gather*}
    |S_{n-1}(1)| <
    \(| \sum_{j=1}^E S_{n-1}(j) \)| <
    \tfrac{yE}{|x|m} \(| \sum_{j=E+1}^{n-2} S_{n-1}(j) \)| < \\
    \frac{\frac{yE}{|x|m}}{\( ( 1 - \frac{yE}{|x|m} \) )}
      \(| \sum_{j=1}^{n-2} S_{n-1}(j) \)| =
    \tfrac{-yE}{yE-|x|m} \(| \sum_{j=1}^{n-2} S_{n-1}(j) \)| \leq
    \tfrac{y}{2} \(| \sum_{j=1}^{n-2} S_{n-1}(j) \)| =
    \tfrac{1}{2}|S_n(0)|.
    \qedhere
  \end{gather*}
\end{proof}

\section*{Acknowledgments}
We wish to give thanks to F.~Przytycki and the Polish Academy of Sciences for
their hospitality at the conference center in B{\k e}dlewo, Poland; to the
Institut Mittag-Leffler for its hospitality during the final phase of this
project; to M.~Benedicks for conversations that led to the project this work
stems from; and to P.~Bleher for suggesting the use of integral operator
theory as a tool for estimating asymptotic growth.

}

\nocite{*}
\bibliographystyle{plain}
\bibliography{catalan}

\end{document}